\documentclass{article}
\usepackage{small_height}

\title{Totally $T$-adic functions of small height}

\author{Xander Faber \\
Center for Computing Sciences \\
Institute for Defense Analyses \\
Bowie, MD \\
\texttt{awfaber@super.org}
\and
Clayton Petsche \\
Oregon State University \\
  Department of Mathematics \\
  Corvallis, OR \\
  \texttt{petschec@math.oregonstate.edu}}

\date{}

\begin{document}
\maketitle

\begin{abstract}
Let $\FF_q(T)$ be the field of rational functions in one variable over a finite
field.  We introduce the notion of a totally $T$-adic function: one that is
algebraic over $\FF_q(T)$ and whose minimal polynomial splits completely over
the completion $\FF_q\Ls{T}$. We give two proofs that the height of a
nonconstant totally $T$-adic function is bounded away from zero, each of which
provides a sharp lower bound.  We spend the majority of the paper providing
explicit constructions of totally $T$-adic functions of small height (via
arithmetic dynamics) and minimum height (via geometry and computer search). We
also execute a large computer search that proves certain kinds of totally
$T$-adic functions of minimum height over $\FF_2(T)$ do not exist. The problem
of whether there exist infinitely many totally $T$-adic functions of minimum
positive height over $\FF_q(T)$ remains open. Finally, we consider analogues of
these notions under additional integrality hypotheses.
\end{abstract}



\section{Introduction}

\subsection{Small totally $T$-adic algebraic functions}

Let $q$ be a prime power, and let $\FF_q(T)$ be the field of rational functions
in one variable over the finite field $\FF_q$ with $q$ elements.  Let
$\overline{\FF_q(T)}$ be an algebraic closure of $\FF_q(T)$.  The (absolute
logarithmic Weil) height is a function $h:\overline{\FF_q(T)}\to\RR$ which
measures the geometric complexity of an algebraic function $\alpha\in
\overline{\FF_q(T)}$.  We will give the precise definition in
\S\ref{sec:ff_heights}, but here we just report that $h(\alpha)\geq0$ for all
$\alpha\in\overline{\FF_q(T)}$, with $h(\alpha)=0$ if and only if $\alpha$ is
constant, or in other words, an element of the algebraic closure of $\FF_q$ in
$\overline{\FF_q(T)}$.  Moreover, if $\alpha=a(T)/b(T)\in \FF_q(T)$ for coprime
polynomials $a(T),b(T)\in\FF_q[T]$, then we have
$h(\alpha)=\max(\deg(a),\deg(b))$.

One of the properties of the height function is that
$h\left(\alpha^{1/n}\right)=\frac{1}{n}h(\alpha)$ for all
$\alpha\in\overline{\FF_q(T)}$ and all natural numbers $n$, and therefore it is
clear that nonconstant algebraic functions of arbitrarily small height exist.
However, the height of a nonconstant algebraic function $\alpha\in
\overline{\FF_q(T)}$ cannot be too small in terms of its degree
$d=[\FF_q(T)(\alpha):\FF_q(T)]$, as it follows immediately from the definition
that $h(\alpha)\geq1/d$.  A comparable bound in the number field setting would
solve a nearly 100 year old (and still open) problem posed by Lehmer
\cite{Lehmer_height_conjecture}. 

In this paper we consider a question on nonconstant algebraic functions of small
height in $\overline{\FF_q(T)}$ which is similarly inspired by the number field
setting, but which turns out to be more fruitful than the question of Lehmer
type.  Let $\FF_q\Ls{T}$ be the fraction field of the ring $\FF_q\Ps{T}$
of formal power series in $T$, and recall that, in a construction similar to
that of the $p$-adic numbers, $\FF_q\Ls{T}$ is the completion of $\FF_q(T)$
with respect to the valuation $\ord_0$ measuring divisibility of rational
functions by $T$.

We say an algebraic function $\alpha\in \overline{\FF_q(T)}$ is \textbf{totally
  $T$-adic} if the minimal polynomial of $\alpha$ over $\FF_q(T)$ splits
completely over $\FF_q\Ls{T}$.  We denote by
\begin{equation*}
\cT_q=\{\alpha\in \overline{\FF_q(T)}\mid \alpha\text{ is totally $T$-adic}\}
\end{equation*}
the set of totally $T$-adic algebraic functions in $\overline{\FF_q(T)}$. In
\S\ref{sec:totally_T_adic} we will show that $\cT_q$ is a field, and moreover
that it is a subfield of the separable closure $\FF_q(T)^\sep$ of $\FF_q(T)$.

A number of authors have studied totally real or totally $p$-adic algebraic
numbers in $\overline{\QQ}$ of small height; examples include Bombieri/Zannier
\cite{Bombieri_Zannier_Heights_Infinite_Extensions_2001}, Fili
\cite{Fili_totally_p-adic}, Petsche/Stacy \cite{Petsche-Stacy}, Pottmeyer
\cite{Pottmeyer_elementary_bound}, Schinzel \cite{Schinzel_conjugates}, and
Smyth \cite{Smyth_totally_real}.  By analogy with these works, one is led to ask
how small the height $h(\alpha)$ can be for nonconstant $\alpha\in \cT_q$, and
whether such a bound can be given which does not depend on the degree
$d=[\FF_q(T)(\alpha):\FF_q(T)]$.  Our first result gives such an inequality,
which is best possible.

\begin{thmA}
  \label{LowerBoundIntro}
If $\alpha\in \cT_q$ is nonconstant, then $h(\alpha)\geq \frac{1}{q+1}$.
Moreover, there exists $\alpha\in \cT_q$ of degree
$[\FF_q(T)(\alpha):\FF_q(T)]=q+1$ and height $h(\alpha)=\frac{1}{q+1}$.
\end{thmA}

We will give two proofs of the lower bound in this statement.  The first proof
is brief and geometric, so it appears in the next part of the introduction. The
second proof is analytic and inspired by a result on heights of totally $p$-adic
algebraic numbers by Pottmeyer \cite{Pottmeyer_elementary_bound}, following an
earlier unpublished argument of the second author. Considering the partition of
the local field $\FF_q\Ls{T}$ into the three regions $\ord_0<0$, $\ord_0=0$, and
$\ord_0>0$, the proportion of the algebraic conjugates of $\alpha$ lying in each
region may be bounded above in terms of the height $h(\alpha)$ and the
(projective) $T$-adic size of the region.  The inequality $h(\alpha)\geq1/(q+1)$
is the average of these three bounds.  The extremal case $h(\alpha)=1/(q+1)$
occurs precisely when equality occurs in each of the three regional
inequalities.  This observation, together with the $\PGL_2(\bar
\FF_q)$-invariance of the height, forces the degree
$d=[\FF_q(T)(\alpha):\FF_q(T)]$ to be a multiple of $q+1$, and also gives strong
constraints on the $T$-adic locations of the conjugates of such $\alpha$, as we
will describe in \S\ref{sec:min_height}.  These constraints motivate the case of
equality described in the statement of Theorem~A.

It is worth contrasting with the situation of small totally $p$-adic algebraic
numbers.  Pottmeyer \cite{Pottmeyer_elementary_bound} has shown for odd $p$ that
$h(\alpha)>\log(p/2)/(p+1)$ whenever $\alpha\in\bar\QQ$ is totally $p$-adic,
nonzero, and not a root of unity.  This is currently the best known lower bound,
but indeed, as indicated, equality in Pottmeyer's bound is not possible.


\subsection{A geometric approach}

Our geometric proof of Theorem~A begins with the observation
that the minimal polynomial $f(x)$ of a nonconstant algebraic function
$\alpha\in \overline{\FF_q(T)}$ over the ring $\FF_q[T]$ may be viewed as a
polynomial in two variables $x$ and $T$. From that vantage point, it defines an
affine plane curve $C' \subset \Spec \FF_q[T,x]$. Write $C_\alpha$ for the
smooth proper curve birational to $C'$.

Set $n = \deg_T(f)$ and $d = \deg_x(f)$. By the definition of the height in
\S\ref{sec:ff_heights}, we see that $h(\alpha) = n/d$. The variables $T$ and $x$
give rise to rational functions on $C_\alpha$, which in turn yield
$\FF_q$-morphisms $C_\alpha \to \PP^1$ of degrees $d$ and $n$, respectively:
\[
\begin{tikzcd}
  C_\alpha \ar[r,"T"] \ar[d,"x"] & \PP^1 \\
  \PP^1 & 
\end{tikzcd}
\]
Since $\alpha$ is totally $T$-adic, the fiber $T = 0$ splits completely into
rational points. Moreover, $\alpha$ is separable
(Corollary~\ref{cor:separable}), so there are precisely $d$ rational points in
the fiber, and we have
\begin{equation}
  \label{eq:pts_upper}
   d \leq \#C_\alpha(\FF_q).
\end{equation}
Every rational point of $C_\alpha$ maps to a rational point of $\PP^1$ under the
morphism $x$. Since $x$ is generically $n$-to-1 and $\#\PP^1(\FF_q) = q+1$, we
find that
\begin{equation}
  \label{eq:pts_lower}
    \#C_\alpha(\FF_q) \leq n(q+1).
\end{equation}
Combining \eqref{eq:pts_upper} and \eqref{eq:pts_lower} shows that
\[
  h(\alpha) \geq \frac{1}{q+1}.
  \]

A strength of this approach to Theorem~A is that it leads to
a geometric interpretation of the condition that a nonconstant $\alpha \in
\cT_q$ has minimal height. Recall that the \textbf{gonality} of an algebraic curve
$X/\FF_q$ is the minimal degree of an $\FF_q$-morphism from $X$ to $\PP^1$.

\begin{thmB}\label{GonalityThmIntro}
Let $\alpha\in \cT_q$ be nonconstant of degree $d=[\FF_q(T)(\alpha):\FF_q(T)]$
and with minimal height $h(\alpha)=1/(q+1)$, and write $C_\alpha$ for the smooth
proper algebraic curve whose function field is $\FF_q(T)(\alpha)$.  Then
\begin{equation*}
\# C_\alpha(\FF_q)=d=n(q+1),
\end{equation*}
where $n$ is the gonality of $C_\alpha$.  Moreover, $n$ is also equal to the
number of algebraic conjugates $\beta$ of $\alpha$ with $\ord_0(\beta)>0$.
\end{thmB}

In fact, a type of converse to Theorem~B (together with its
preceding discussion) holds.  If $X/\FF_q$ is any smooth, proper, geometrically
irreducible curve with $\# X(\FF_q)=n(q+1)$ for some $n\geq1$, and if
$\FF_q(X)=\FF_q(u,v)$ for separable rational functions $u$ and $v$ satisfying
the expected properties, then in fact $X$ is isomorphic to $C_\alpha$ for some
nonconstant $\alpha\in \cT_q$ of minimal height $h(\alpha)=1/(q+1)$.  The
details are given in \S\ref{sec:geometry}, along with estimates on the genera of
these special curves.


\subsection{Cases of minimal height}

In view of the preceding results, it is natural to ask:

\begin{question}\label{OpenQuestionInfinite}
For a given prime power $q$, do there exist infinitely many nonconstant $\alpha\in \cT_q$ of minimal height $h(\alpha)=\frac{1}{q+1}$?  
\end{question}

This question is still open for every $q$.  Heuristic considerations based on
the paper \cite{BGMR_splitting} suggest that the answer may be ``yes'' when $q >
2$, and ``no'' when $q=2$.

By Northcott's theorem, Question \ref{OpenQuestionInfinite} is tantamount to
asking whether there exist nonconstant $\alpha\in \cT_q$, of minimal height
$h(\alpha)=\frac{1}{q+1}$, and of arbitrarily large degree
$d=[\FF_q(T)(\alpha):\FF_q(T)]$.  And according to
Theorem~B, yet another formulation of this question is to
ask whether the curves $C_\alpha/\FF_q$ can have arbitrarily large gonality for
nonconstant $\alpha\in \cT_q$ with $h(\alpha)=1/(q+1)$.  Thus the following may
be viewed as a refinement of Question \ref{OpenQuestionInfinite}.

\begin{question}\label{OpenQuestionParticularGonality}
For a given prime power $q$ and integer $n\geq1$, does there exist a nonconstant
$\alpha\in \cT_q$ of minimal height $h(\alpha)=\frac{1}{q+1}$ such that the
curve $C_\alpha$ has gonality $n$?
\end{question}

The case of equality described in Theorem~A gives an
affirmative answer to the $n=1$ case of this question for all $q$.  The
construction is elementary and the associated algebraic curves $C_\alpha$ are
all rational, so no interesting geometry is involved.

In \S\ref{sec:examples} we give geometric constructions to go beyond the $n=1$
case. Using cyclic $n$-covers of the projective line we produce examples
affirming Question~\ref{OpenQuestionParticularGonality} when $n\mid q-1$, and we
use hyperelliptic curves to produce examples when $n = 2$ and $q\geq4$ is even.

None of the general constructions described in \S\ref{sec:examples} yields an
affirmative answer to Question~\ref{OpenQuestionParticularGonality} for $n\geq
q$.  However, in \S\ref{sec:naive_algorithm} we describe an algorithm which uses
the ideas from the analytic proof of Theorem~A to
search for examples, and the algorithm succeeds in confirming a positive answer
to Question~\ref{OpenQuestionParticularGonality} in the new case $q=3$ and $n=3$.

In what may be an indication of the general difficulty of
Question~\ref{OpenQuestionParticularGonality}, there do exist $q$ and $n$ for
which the answer is ``no''.  In \S\ref{sec:non-existence} we describe an
exhaustive search algorithm which gives a negative answer to
Question~\ref{OpenQuestionParticularGonality} in the cases $q=2$ and
$n=2,3,4$. We used Sage \cite{sagemath.9.1} to carry out these calculations. The
case $q = 2$ and $n = 4$ was a substantial computational challenge. Our code is
available at
\begin{center}
  \url{https://github.com/RationalPoint/T-adic}.
\end{center}


\subsection{Small height and large degree via arithmetic dynamics}

As we have seen in the preceding section, the only known cases of equality in
the sharp lower bound $h(\alpha)\geq 1/(q+1)$ for nonconstant $\alpha\in\cT_q$
occur for $\alpha$ of degree $d=n(q+1)$ with $n\leq q$.  In particular, there
are only finitely many known cases of equality (for each given prime power $q$)
by Northcott's theorem.  Thus it is reasonable to ask whether a larger lower
bound is possible if one is willing to ignore some finite subset of $\cT_q$.  We
do not know the answer to this question, but we can show that no lower bound
greater than $1/(q-1)$ is possible, even when one is allowed to remove any
finite subset of $\cT_q$ from consideration.

Let's be more precise. Given a subfield $L$ of $\overline{\FF_q(T)}$, choose
an ordering $L = \{\alpha_1, \alpha_2, \ldots\}$ and define
\[
\liminf_{\alpha \in L} h(\alpha) = \lim_{n \to \infty} \inf \big\{ h(\alpha_i) \ : \ i \geq n \big\}.
\]
It is independent of the chosen ordering. 

\begin{thmC}\label{LimInfIntro}
For each prime power $q$, we have
\begin{equation}\label{LimInfIntroBounds}
\frac{1}{q+1} \leq \liminf_{\alpha\in \cT_q}h(\alpha) \leq \frac{1}{q-1}.
\end{equation}
\end{thmC}

The lower bound in \eqref{LimInfIntroBounds} is immediate from Theorem~A, as the
only $\alpha\in\cT_q$ with $h(\alpha)<1/(q+1)$ are the elements of $\FF_q$.  To
prove the upper bound in \eqref{LimInfIntroBounds}, we use a dynamical
construction inspired by a result of Petsche-Stacy \cite{Petsche-Stacy} on
totally $p$-adic algebraic numbers, which was in turn inspired by work of Smyth
\cite{Smyth_totally_real} in the totally real case.

The construction makes use of the polynomial
$\phi(x)=\frac{1}{T}(x^q-x)\in\FF_q(T)[x]$, which restricts to a $q$-to-$1$ map
from the local ring $\FF_q\Ps{T}$ onto itself.  Thus for each $m\geq1$, the
iterated inverse image $\phi^{-m}(1)$ of the constant $1$ is a set of $q^m$
distinct totally $T$-adic algebraic functions.  Defining recursively
$\alpha_0=1$ and $\phi(\alpha_{j+1})=\alpha_j$, we can show that
$h(\alpha_j)\to1/(q-1)$ as $j\to\infty$.  The details are worked out in
\S\ref{sec:small_height}.

\begin{question}\label{OpenQuestionLiminf}
What is the value of $\liminf_{\alpha\in\cT_q}h(\alpha)$?  Is it equal to the
minimum height $\frac{1}{q+1}$ for nonconstant $\alpha\in\cT_q$?
\end{question}

We do not know the answer to this question.  Note that an affirmative answer to
Question~\ref{OpenQuestionInfinite} would imply an affirmative answer to
Question~\ref{OpenQuestionLiminf}.  In the setting of totally real algebraic
integers $\alpha$, the smallest positive height occurs for the golden ratio
$\alpha = \frac{1+\sqrt{5}}{2}$ (a result of Schinzel
\cite{Schinzel_conjugates}), but Smyth has shown that there is a gap between
this minimum positive height and the liminf.


\subsection{Totally $T$-adic integers and totally $T$-adic units}

The ideas behind Theorem~A and Theorem~C can be extended to give stronger
conclusions for $\alpha\in\cT_q$ whose conjugates satisfy $T$-adic integrality
restrictions.  Define a subring $\cR_q$ of $\cT_q$ by
\[
\cR_q = \{\alpha\in\cT_q\mid \ord_0(\beta)\geq0\text{ for all conjugates $\beta\in\FF_q\Ls{T}$ of $\alpha$}\}
\]
and consider its group of units
\[
\cR_q^\times = \{\alpha\in\cT_q\mid \ord_0(\beta)=0\text{ for all conjugates $\beta\in\FF_q\Ls{T}$ of $\alpha$}\}.
\]
We call $\cR_q$ the ring of \textbf{totally $T$-adic integers} in
$\overline{\FF_q(T)}$, and $\cR_q^\times$ the group of \textbf{totally $T$-adic
  units}.

The following results may be compared to Theorem~A and Theorem~C, respectively.

\begin{thmD}\label{IntUnitsLBIntro}
Let $q$ be a prime power.  If $\alpha\in \cR_q$ is nonconstant, then
$h(\alpha)\geq \frac{1}{q}$.  If $\alpha\in \cR_q^\times$ is nonconstant, then
$h(\alpha)\geq \frac{1}{q-1}$.  Both inequalities are sharp.
\end{thmD}

\begin{thmE}\label{IntUnitsLimInfIntro}
Let $q$ be a prime power.  We have
\begin{equation*}
\frac{1}{q} \leq \liminf_{\alpha\in\cR_q}h(\alpha) \leq \frac{1}{q-1}
\end{equation*}
and 
\begin{equation*}
\frac{1}{q-1} \leq \liminf_{\alpha\in\cR_q^\times}h(\alpha) \leq 
\begin{cases}
\frac{1}{q-2} & \text{ if } q\neq2 \\
2 & \text{ if } q=2.
\end{cases}
\end{equation*}
\end{thmE}


\section{Preliminaries}

\subsection{Function field heights}
\label{sec:ff_heights}

Let $\alpha$ be algebraic over $\FF_q(T)$. There is a unique monic irreducible
polynomial $f_0 \in \FF_q(T)[x]$ such that $f_0(\alpha) = 0$. After clearing
denominators of the coefficients and perhaps multiplying by an appropriate
element of $\FF_q$, we arrive at an irreducible polynomial $f \in \FF_q[T][x]$
such that:
\begin{itemize}
\item $f$ is irreducible as a polynomial in $x$;
\item $f(\alpha) = 0$;  
\item The coefficients of $f$ are polynomials in $\FF_q[T]$ with no common factor; and
\item The leading coefficient of $f$ is monic.
\end{itemize}
The unique polynomial with these properties will be called the \textbf{minimal
  polynomial of $\alpha$}.

Take $\alpha$ and $f$ as in the preceding paragraph. We define the
\textbf{(absolute) height} of $\alpha$ to be
\[
  h(\alpha) = \frac{\deg_T(f)}{\deg_x(f)}.
  \]
We also define $h(\infty) = 0$ in order to have a definition of the height on
all of $\PP^1(\overline{\FF_q(T)})$.

\begin{example}
  Let $\alpha \in \FF_{q^d}^\times$. Then the minimal polynomial of $\alpha$ is a
  divisor of $x^{q^d} - x$. In particular, it does not depend on $T$, so that
  $h(\alpha) = 0$. 
\end{example}

We now reformulate the height as a sum of local contributions. Before doing so,
we need a lemma that relates the valuations of the ``large'' roots of a
polynomial to the valuation of its leading coefficient. 

\begin{lemma}
  Let $(R,v)$ be a discrete valuation ring, and let $g(x) = g_d x^d + \cdots +
  g_0 \in R[x]$ be a separable polynomial such that $g_d \ne 0$ and
  $\min\big(v(g_i) : i \geq 0\big) = 0$. Choose an extension of $v$ to the
  splitting field of $g$. Write $\beta_1, \ldots, \beta_d$ for the roots of
  $g$. Then
  \begin{equation}
    \label{eq:valuation}
    v(g_d) = \sum_{i=1}^d \max\big(0,-v(\beta_i)\big).
  \end{equation}
  If $v(g_d) = 0$, then all terms in the sum vanish. 
\end{lemma}

\begin{proof}
  If $v(g_d) = 0$, then the Newton polygon for $g$ shows that all of its roots
  have nonnegative valuation. Consequently, the sum in \eqref{eq:valuation}
  vanishes.

  Now suppose that $v(g_d) > 0$.  Let $\ell$ be the largest index such that
  $v(g_\ell) = 0$. Without loss of generality, let us suppose that the roots of
  $g$ are ordered by valuation as
  \[
  v(\beta_d) \leq v(\beta_{d-1}) \leq \cdots \leq v(\beta_1).
  \]
  Looking at the Newton polygon for $g$, we see that
  \[
  v(\beta_{\ell+1}) < 0 \le v(\beta_\ell).
  \]
  The coefficient $g_\ell$ can be expressed in terms of the roots as
  \[
  g_\ell = \pm g_d \sum_{i_1 < \cdots < i_{d - \ell}} \beta_{i_1} \cdots \beta_{i_{d-\ell}}.
  \]
  The unique term in this sum with smallest valuation is $\beta_{\ell+1} \cdots
  \beta_d$, so taking valuations of both sides gives
  \[
  0 = v(g_\ell) = v(g_d) + \sum_{i = \ell+1}^d v(\beta_i)
  = v(g_d) + \sum_{i=1}^d \min\big(0,v(\beta_i)\big).
  \]
  This is equivalent to \eqref{eq:valuation}.
\end{proof}

The closed points of the projective line $\PP^1_{\FF_q}$, denoted
$|\PP^1_{\FF_q}|$, are in bijective correspondence with the places of
$\FF_q(T)$; to a point $Q$, we associate the order of vanishing of a function at
$Q$, denoted $\ord_Q$. If $Q = \infty$, then $\ord_Q(f) = -\deg(f)$ for any
polynomial $f \in \FF_q[T]$. If $Q \ne \infty$, then $Q$ may be identified with
a unique monic irreducible polynomial $f_Q$ in $\FF_q[T]$, and $\ord_Q$
corresponds to the order to which $f_Q$ divides a rational function in
$\FF_q(T)$. Each of the functions $\ord_Q$ is a discrete valuation. We abuse
notation and continue to write $\ord_Q$ for a fixed extension of the valuation
to $\overline{\FF_q(T)}$.

\begin{proposition}
  Let $\alpha \in \overline{\FF_q(T)}$ have minimal polynomial $f \in
  \FF_q[T][x]$, and suppose that $\deg(f) = d$. Write $\alpha = \alpha_1,
  \ldots, \alpha_d$ for the roots of $f$ (with appropriate multiplicity if $f$
  is inseparable). Then
\[
h(\alpha) = \frac{1}{d} \sum_{Q \in |\PP^1_{\FF_q}|} \ \sum_{i=1}^d \max\big(0, -\ord_Q(\alpha_i)\big),
\]
where all but finitely many terms in this sum vanish.
\end{proposition}

\begin{proof}  
  Let $f = a_d x^d + \cdots + a_0$ be the minimal polynomial of $\alpha$, so
  that the coefficients lie in $\FF_q[T]$ and have no common factor. Fix $Q \ne
  \infty$. Then $\ord_Q$ measures divisibility by a certain monic irreducible
  polynomial, and hence $\min\big(\ord_Q(a_i) : i \geq 0\big) = 0$. We apply the preceding lemma to deduce that
  \begin{equation}
    \label{eq:local1}
  \ord_Q(a_d) = \sum_{i=1}^d \max\big(0, -\ord_Q(\alpha_i)\big),
  \end{equation}
  and if $\ord_Q(a_d) = 0$, then all terms in this sum vanish.

  Now consider the case $Q = \infty$. For a polynomial $g \in \FF_q[T]$, we have
  $\ord_\infty(g) = -\deg_T(g)$. Let $n = \deg_T(f)$. Then
  \[
  \min \big(\ord_\infty(a_i T^{-n}) : i \geq 0\big) = 0.
  \]
  So we may apply the preceding lemma to $T^{-n} f$ to arrive at
  \begin{equation}
    \label{eq:local2}
  n + \ord_\infty(a_d) = \sum_{i=1}^d \max\big(0, -\ord_\infty(\alpha_i)\big).
  \end{equation}
  Adding \eqref{eq:local1} over all $Q \neq \infty$ and \eqref{eq:local2} gives
  \[
  \deg_T(f) + \sum_Q \ord_Q(a_n) = \sum_{Q \in |\PP^1_{\FF_q}|} \ \sum_{i=1}^d \max\big(0, -\ord_Q(\alpha_i)\big).
  \]
  The divisor of a function has total degree~0, so $\sum_Q \ord_Q(a_n) =
  0$. Dividing by $d = \deg_x(f)$ gives the result.
\end{proof}

At times, it is more convenient to have a formulation of the height of $\alpha$
that requires less up-front knowledge of its minimal polynomial:

\begin{corollary}
  \label{cor:field_embeddings}
  Let $\alpha \in \overline{\FF_q(T)}$ be separable over $\FF_q(T)$. For any
  separable extension $K / \FF_q(T)$ containing $\alpha$, we have
\[
h(\alpha) = \frac{1}{[K:\FF_q(T)]} \sum_{Q \in |\PP^1_{\FF_q}|} \ \sum_{\sigma : K \hookrightarrow \overline{\FF_q(T)}} \max\big(0, -\ord_Q\left(\sigma(\alpha)\right)\big),
\]
where the inner sum is over the field embeddings of $K$ into
$\overline{\FF_q(T)}$ that fix $\FF_q(T)$ pointwise.
\end{corollary}

The proof is similar to the number field case, so we omit it.

\begin{lemma}
  \label{lem:ht_facts}
  Let $\alpha,\beta$ be algebraic over $\FF_q(T)$ with $\alpha \ne 0$, and let
  $n$ be an integer. Then
  \begin{itemize}
    \item $h(\alpha^n) = |n| \ h(\alpha)$;
    \item $h(\alpha + \beta) \leq h(\alpha) + h(\beta)$; and
    \item $h(\alpha\beta) \leq h(\alpha) + h(\beta)$.
  \end{itemize}
\end{lemma}

\begin{proof}
The proofs of all three statements proceed exactly as in the case of heights on
algebraic number fields; see, for example,
\cite[\S1.5]{Bombieri-Gubler_2006}. The second inequality for number fields has
a $\log 2$ term in it that arises from the Archimedean places; evidently, we can
dispense with that in the function field case.
\end{proof}

\begin{remark}
  In light of the fact that $h(\alpha) = h(\alpha^{-1})$ for $\alpha \ne 0$, we
  can drop the sign in our formulation of the height:
  \[
  h(\alpha) = \frac{1}{d} \sum_{Q \in |\PP^1_{\FF_q}|} \ \sum_{i=1}^d \max\big(0, \ord_Q(\alpha_i)\big),
  \]
  and similarly for the formulation involving field embeddings from
  Corollary~\ref{cor:field_embeddings}. 
\end{remark}

\begin{lemma}
  \label{lem:invariant}
  The height is invariant under the action of $\PGL_2(\bar \FF_q)$. That
  is, if $\alpha$ is algebraic over $\FF_q(T)$ and $\gamma \in \PGL_2(\bar\FF_q)$,
  then $h(\gamma(\alpha)) = h(\alpha)$.
\end{lemma}

\begin{proof}
The group $\PGL_2(\bar \FF_q)$ is generated by the elements
\begin{eqnarray*}
  \gamma_0(z) &=& 1/z, \\
  \gamma_1(z) &=& z + t, \text{ where } t \in \bar \FF_q \\
  \gamma_2(z) &=& uz, \text{ where } u \in \bar \FF_q^\times.
\end{eqnarray*}
Consequently, it suffices to verify that 
$\alpha$, $1/\alpha$, $\alpha+t$, and $u\alpha$ have the same height. This is
clear if $\alpha = \infty$, so we assume $\alpha$ is finite for the remainder of
the proof.

Clearly $h(1/\alpha) = h(\alpha)$ by Lemma~\ref{lem:ht_facts}. Again using
Lemma~\ref{lem:ht_facts}, we have
\[
h(\alpha) = h(\alpha + t - t) \leq h(\alpha+t) + h(-t) 
= h(\alpha+t) \leq h(\alpha) + h(t) =
h(\alpha).
\]
Thus every inequality in the chain is an equality, and $h(\alpha+t) =
h(\alpha)$.
Similarly,
\[
h(\alpha) = h\big(u\alpha(1/u)\big) \leq h(u\alpha) + h(1/u) = h(u\alpha) \leq
h(u) + h(\alpha) = h(\alpha)
\]
proves that $h(u\alpha) = h(\alpha)$.
\end{proof}


\subsection{Totally $T$-adic functions}
\label{sec:totally_T_adic}

Recall that a function $\alpha \in \overline{\FF_q(T)}$ is totally $T$-adic if
its minimal polynomial splits completely over $\FF_q\Ls{T}$.

\begin{proposition}
  Let $L \subset \FF_q\Ls{T}$ be a subfield of the Laurent series field over
  $\FF_q$. Suppose that $L$ is algebraic over $\FF_q(T)$. Then $L$ is a
  separable extension of $\FF_q(T)$.
\end{proposition}

\begin{proof}
Let $L^\sep / \FF_q(T)$ be the maximal separable subextension of $L / \FF_q(T)$;
then $L / L^\sep$ is purely inseparable. We claim that $L / L^\sep$ is trivial. 

Suppose that $L / L^\sep$ is nontrivial, and let $\beta \in L \smallsetminus
L^\sep$. Then there exists $b \in L^\sep$ and $n \geq 1$ such that $\beta^{p^n}
= b$, where $p$ be the characteristic of $\FF_q$. Note that $b$ is not a
$p^n$-th power in $L^\sep$, else $\beta \in L^\sep$. Since $b \in \FF_q\Ls{T}$,
we may write
\[
b = \sum_{j \in \ZZ} c_j T^j,
\]
where $c_j \in \FF_q$ and $c_j = 0$ for all $j$ sufficiently negative. Write
$c'_j$ for the unique element of $\FF_q$ such that $(c'_j)^{p^n} = c_j$. Then
the unique root of $x^{p^n} - b$ inside $\overline{\FF_q\Ls{T}}$ is given by
\[
\beta = \sum_{j \in \ZZ} c'_j (T^{1/p^n})^j.
\]

We now argue that $\beta$ is not an element of $\FF_q\Ls{T}$, contrary to the fact
that $L \subset \FF_q\Ls{T}$. For suppose otherwise, and consider the set of
indices $S = \{j : c_j \ne 0\}$. If every element of $S$ is a multiple of $p^n$,
then $b$ is a $p^n$-th power, a contradiction. So there is a minimum index $m
\in S$ that is not a $p^n$-th power. Write
\[
\beta_0 = \sum_{\substack{j \in S\\ p^n \mid j}} c'_j (T^{1/p^n})^j = \sum_{p^n j \in S} c'_{p^nj} T^j \in \FF_q\Ls{T}.
\]
Now we see that
\[
\beta - \beta_0 = c'_m (T^{1/p^n})^m + \text{ higher order terms}.
\]
If $\beta \in \FF_q\Ls{T}$, then so is $\beta - \beta_0$, and hence its $T$-adic
valuation is an integer. But we also see that $\ord_0(\beta - \beta_0) =
\frac{m}{p^n} \not\in \ZZ$. This contradiction completes the proof.
\end{proof}

\begin{corollary}
  \label{cor:separable}
  Let $\alpha \in \overline{\FF_q(T)}$ be a totally $T$-adic function. Then
  $\alpha$ is separable over $\FF_q(T)$.
\end{corollary}

\begin{proof}
Let $f$ be the minimal polynomial for $\alpha$. Since $\alpha$ is totally
$T$-adic, there is a splitting field $L \subset \FF_q\Ls{T}$ for $f$. We may
assume that $L$ is of finite dimension over $\FF_q(T)$. By the preceding
proposition, we see that $L$ is separable over $\FF_q(T)$, and hence so is
$\alpha$.
\end{proof}

\begin{proposition}
  \label{prop:Tq_is_a_field}
  The set $\cT_q$ of totally $T$-adic functions is a
  field.
\end{proposition}

\begin{proof}
  Let $\alpha \in \cT_q$, and let $f$ be its minimal polynomial. If $\alpha \ne
  0$, then $f^*(x) = x^{\deg(f)}f(1/x)$ is the minimal polynomial for
  $\alpha^{-1}$. Evidently, $f^*$ splits completely over $\FF_q\Ls{T}$ because
  $f$ does, so $\alpha^{-1} \in \cT_q$.

  Suppose that $\alpha, \beta \in \cT_q$ have minimal polynomials $f, g$,
  respectively. Write $\alpha_1, \ldots, \alpha_m \in \FF_q\Ls{T}$ for the roots
  of $f$, and write $\beta_1, \ldots, \beta_n \in \FF_q\Ls{T}$ for the roots of
  $g$. Consider the polynomial
  \[
    F(x) = \prod_{i=1}^m \prod_{j=1}^n (x - \alpha_i - \beta_j).
    \]
  The coefficients of $F$ are symmetric in the sums $\alpha_i + \beta_j$. Since
  totally $T$-adic elements are separable, the absolute Galois group of
  $\FF_q(T)$ permutes the set of such sums. Hence, $F$ has coefficients in
  $\FF_q(T)$. If we write $P$ for the minimal polynomial of $\alpha + \beta$, it
  follows that $P$ divides $F$. By construction, $F$ splits completely over
  $\FF_q\Ls{T}$, and thus, so does $P$. We conclude that $\alpha + \beta \in
  \cT_q$. A similar argument applies to the product $\alpha \beta$.
\end{proof}

\begin{proposition}
  The set $\PP^1(\cT_q) = \cT_q \cup \infty$ is preserved by the action of
  $\PGL_2(\FF_q)$
\end{proposition}

\begin{proof}
  This is immediate from the fact that $\FF_q$ is a subfield of the field $\cT_q$.
\end{proof}


\section{Minimum-height functions}
\label{sec:min_height}

We have two goals for this section. The first is to study the minimal
polynomials of totally $T$-adic functions of minimum height. In so doing, we
will give an analytic proof of our main height lower bound:

\begin{theorem}
  \label{thm:height_bound}
  Let $\alpha$ be a nonconstant totally $T$-adic function. Then $h(\alpha) \geq
  \frac{1}{q+1}$.
\end{theorem}

Then we will turn to the geometry of the algebraic curves associated to totally
$T$-adic functions of minimum height. Our study of minimal polynomials will
provide insight into the singularities of plane models of these curves.


\subsection{Minimal polynomials}
\label{sec:min_pols}

We begin with a pair of lemmas that help us understand the height in terms of
the locations of the conjugates of $\alpha$ inside $\FF_q\Ls{T}$.

\begin{lemma}
  \label{lem:positive-valuation}  
Let $\alpha \ne 0$ be totally $T$-adic of degree~$d$ over $\FF_q(T)$, and
suppose that $r$ of its Galois conjugates have positive $T$-adic valuation. Then
\[
h(\alpha) \geq \frac{r}{d}.
\]
\end{lemma}

\begin{proof}
Write $K = \FF_q(T,\alpha)$. Since $\alpha$ is totally $T$-adic, any conjugate
$\beta$ with positive $T$-adic valuation necessarily satisfies $\ord_0(\beta)
\geq 1$. Therefore,
  \begin{align*}
h(\alpha) &= \frac{1}{d}\sum_{Q \in |\PP^1_{\FF_q}|} \ \ 
\sum_{\sigma \colon K \hookrightarrow \overline{\FF_q(T)}} \max\left(0,
\ord_Q\left(\sigma(\alpha)\right)\right) \\
& \geq \frac{1}{d} \sum_{\sigma \colon K \hookrightarrow \overline{\FF_q(T)}} \max\left(0,
\ord_0\left(\sigma(\alpha)\right)\right) \geq \frac{r}{d}. \qedhere
  \end{align*}
\end{proof}

\begin{lemma}
  \label{lem:zero-valuation}
Let $\alpha \not\in \FF_q^\times$ be totally $T$-adic of degree~$d$ over $\FF_q(T)$,
and suppose that $\ell$ of its Galois conjugates have $T$-adic valuation
zero. Then
\[
h(\alpha) \geq \frac{\ell}{d(q-1)}.
\]
\end{lemma}

\begin{proof}
Since $\alpha$ is totally $T$-adic, any conjugate $\beta$ with $T$-adic
valuation~0 satisfies
\[
\ord_0(\beta^{q-1} - 1) \geq 1.
\]
Writing $K = \FF_q(T,\alpha)$, we see that
\begin{align*}
  (q-1)h(\alpha) = h\big(\alpha^{q-1}\big) &= h\big( \alpha^{q-1} - 1\big) \\
&= \frac{1}{d}\sum_{Q \in |\PP^1_{\FF_q}|} \ \
\sum_{\sigma \colon K \hookrightarrow \overline{\FF_q(T)}} \max\left(0,
\ord_Q\left(\sigma\left(\alpha^{q-1} - 1\right)\right)\right) \\
&\geq \frac{1}{d} \sum_{\substack{\sigma \colon K \hookrightarrow \overline{\FF_q(T)}
       \\ \ord_0(\sigma(\alpha)) = 0}}
\max\left(0,
  \ord_0\left(\sigma\left(\alpha^{q-1} - 1\right)\right)\right) \\
  &\geq \frac{1}{d} \sum_{\substack{\sigma \colon K \hookrightarrow \overline{\FF_q(T)}
       \\ \ord_0(\sigma(\alpha)) = 0}} 1 = \frac{\ell}{d}.
\end{align*}
This is equivalent to the desired inequality.
\end{proof}

\begin{proof}[Analytic proof of Theorem~\ref{thm:height_bound}]
Write $r$ for the number of Galois conjugates of $\alpha$ with positive $T$-adic
valuation, and write $s$ for the number of conjugates with negative
valuation. If $d$ is the degree of $\alpha$ over $\FF_q(T)$, then the number of
conjugates with valuation zero is $d - r - s$.

Lemma~\ref{lem:positive-valuation} applied to $\alpha$ shows that
\[
h(\alpha)\geq \frac{r}{d}.
\]
Applying the same lemma to $\alpha^{-1}$ yields
\[
h(\alpha)\geq \frac{s}{d}.
\]
Finally, applying Lemma~\ref{lem:zero-valuation} gives us
\[
(q-1)h(\alpha) \geq \frac{d - r - s}{d}.
\]
Summing the last three displayed inequalities gives $(q+1)h(\alpha) \geq 1$.
\end{proof}

By studying this proof, we can deduce a number of properties that a polynomial
$f \in \FF_q[T][x]$ must have in order to be the minimal polynomial of a totally
$T$-adic function with height $\frac{1}{q+1}$.

\begin{proposition}
  \label{prop:newton_criterion}
Let $\alpha \in \cT_q$ have height $\frac{1}{q+1}$, and let $f \in \FF_q[T][x]$ be
its minimal polynomial. Write $r$ for the number of roots of $f$ with positive
$T$-adic valuation. Then the following hold:
\begin{enumerate}
\item $\deg_x(f) = r(q+1)$ and $\deg_T(f) = r$;
\item For each $u \in \FF_q$, the Newton polygon of $f(x+u)$ with respect to
  $\ord_0$ is the lower convex hull of the points $(0,r)$, $(r,0)$, $(rq,0)$,
    and $(r(q+1),r)$;
\item The leading coefficient of $f$ is $T^r$; and
\item For each $u \in \FF_q$, the constant coefficient of $f(x+u)$ is of the
  form $aT^r$ for some $a \in \FF_q^\times$.
\end{enumerate}
\end{proposition}

\begin{proof}
In order to have $h(\alpha) = \frac{1}{q+1}$, it is necessary that equality is
attained in all of the inequalities in the proof of
Theorem~\ref{thm:height_bound}. As in the proof, we take $r$ (resp. $s$) to be
the number of conjugates of $\alpha$ whose $T$-adic valuation is positive
(resp. negative). We let $d = \deg_x(f)$. In particular, the following
equalities hold:
\[
\frac{1}{q+1} = \frac{\deg_T(f)}{\deg_x(f)} = h(\alpha) = \frac{r}{d} = \frac{s}{d}.
\]
We see immediately that $d = r(q+1)$ and $\deg_T(f) = r$, which proves the first
assertion of the proposition. We also find that $r = s$.

Looking at the proof of Lemma~\ref{lem:positive-valuation}, we see that each
conjugate $\beta$ with positive valuation must have $\ord_0(\beta) = 1$. It
follows that the left-most segment of the Newton polygon of $f$ has vertices
$(0,r)$ and $(r,0)$. Similarly, since the equality $h(\alpha) = \frac{s}{d}$
holds, each conjugate $\beta$ with negative valuation must have $\ord_0(\beta) =
-1$. Hence, the right-most segment of $f$ has vertices $(d-s,0) = (rq,0)$ and
$(d,s) = (r(q+1),r)$. As $f$ has $T$-adically integral coefficients, we have
found all of the vertices of the Newton polygon for $f$. This proves the second
assertion in the case $u = 0$.

To obtain the full power of the second assertion, fix $u \in \FF_q$ and observe
that $\alpha - u$ also has height $1/(q+1)$. The preceding paragraphs apply to
the minimal polynomial of $\alpha - u$, namely $f(x+u)$.

For the third assertion, let $c \in \FF_q[T]$ be the leading coefficient of
$f$. It is monic by the definition of the minimal polynomial, and it is
divisible by $T^r$ by the second assertion. If $c$ were divisible by any
irreducible factor different from $T$, then its degree would necessarily be
larger than $r$, contradicting the first assertion.

Finally, we look at the fourth assertion. By the second assertion, the
polynomial $f(x+u)$ has constant coefficient divisible by $T^r$. If the constant
coefficient were divisible by any irreducible factor different from $T$, then
its degree would be larger than $r$, contradicting the first assertion. (Note
that the first assertion also applies to $f(x+u)$ because it is the minimal
polynomial of $\alpha - u$, which has height $1/(q+1)$.)
\end{proof}

\begin{corollary}[Well-distributed Roots]
  \label{cor:well-distributed}
Let $\alpha \in \cT_q$ have height $\frac{1}{q+1}$; let $f \in \FF_q[T][x]$ be
its minimal polynomial; and let $r$ be the number of roots of $f$ with positive
$T$-adic valuation. For each $u \in \PP^1(\FF_q)$, $f$ has $r$ distinct roots
whose reduction modulo $T$ is $u$. Moreover, the following hold:
\begin{enumerate}
\item A root $\beta$ of $f$ that reduces to $u \in \FF_q$ satisfies
  $\ord_0(\beta - u) = 1$.
\item A root $\beta$ of $f$ that reduces to $\infty$ satisfies
  $\ord_0(\beta^{-1}) = 1$. 
\end{enumerate}
\end{corollary}

\begin{proof}
  The Newton polygon statement in Proposition~\ref{prop:newton_criterion}
  implies that $f(x+u)$ has exactly $r$ roots with positive $T$-adic
  valuation. In particular, it shows that $f$ has precisely $r$ roots that
  reduce to $-u \in \FF_q$, independent of $u$. The Newton polygon for $f$ also
  shows that it has $r$ roots that reduce to $\infty$.

  Suppose that $\beta$ is a root of $f$ with positive $T$-adic valuation.  The
  proof of Proposition~\ref{prop:newton_criterion} shows that
  $\ord_0(\beta) = 1$, lest we violate the equality $h(\alpha) = r/d$.

  Next suppose that $\beta$ is a root of $f$ that reduces to $u \in \FF_q$.  Then
  $\beta - u$ is a root of $f(x+u)$, which is also the minimal polynomial of an
  element of height $\frac{1}{q+1}$. Since $\beta - u$ has positive valuation, the
  preceding paragraph shows that $\ord_0(\beta-u) = 1$.

  Finally, suppose that $\beta$ is a root of $f$ that reduces to $\infty$. Then
  $\beta^{-1}$ reduces to $0$, and it is a root of $f^*$, the reversed polynomial,
  which is also the minimal polynomial of a totally $T$-adic element of height
  $\frac{1}{q+1}$.
  Just as above, we conclude that $\ord_0(\beta^{-1}) = 1$.
\end{proof}


\subsection{Geometry and arithmetic of associated curves}
\label{sec:geometry}

Recall that the \textbf{gonality} of an algebraic curve $X_{/\FF_q}$ is the
minimum degree of a nonconstant $\FF_q$-morphism to $\PP^1$. Each function
$\alpha \in \overline{\FF_q(T)}$ gives rise to an algebraic curve with function
field $\FF_q(T,\alpha)$; we define the gonality of $\alpha$ to be the gonality
of this curve.

\begin{theorem}
  \label{thm:point_count}
  Let $\alpha$ be a totally $T$-adic function such that $h(\alpha) =
  \frac{1}{q+1}$, and let $C_\alpha$ be its associated algebraic curve. Then
  $\#C_\alpha(\FF_q) = n(q+1)$, where $n$ is the gonality of $C_\alpha$.
\end{theorem}

\begin{proof}
  Since $h(\alpha) = \frac{1}{q+1}$, both inequalities \eqref{eq:pts_lower} and
  \eqref{eq:pts_upper} are attained, which yields
  \[
  d = \#C_\alpha(\FF_q) = n(q+1),
  \]
  where $n = \deg(x) = \deg_T(f)$. It remains to show that $n$ is the gonality
  of $C_\alpha$. Suppose that $\psi : C_\alpha \to \PP^1$ is any nonconstant
  morphism. Then  the same argument that gave us \eqref{eq:pts_lower} would
  show that $\#C_\alpha(\FF_q) \leq \deg(\psi)(q+1)$. Hence, $n \leq
  \deg(\psi)$. That is, $n$ is the minimum degree of a morphism to $\PP^1$.
\end{proof}

\begin{corollary}
  \label{cor:r=n}
  Let $\alpha$ be a totally $T$-adic function of height $\frac{1}{q+1}$. The
  gonality of $\alpha$ is equal to the number of Galois conjugates of $\alpha$
  with positive $T$-adic valuation.
\end{corollary}

\begin{proof}
Let $f$ be the minimal polynomial of $\alpha$. Write $r$ for the number of its
roots with positive $T$-adic valuation. Write $n$ for the gonality of $\alpha$.
On one hand, Proposition~\ref{prop:newton_criterion} shows that $\deg_x(f) =
r(q+1)$. On the other, the preceding theorem says that $\deg_x(f) = \#C_\alpha(\FF_q) =
n(q+1)$.
\end{proof}

To summarize, when $\alpha$ has height $\frac{1}{q+1}$, the associated
curve $C_\alpha$ has the following properties:
\begin{itemize}
\item $\#C_\alpha(\FF_q) = n(q+1)$ for some $n \geq 1$;
\item $C_\alpha$ admits rational functions $T$ and $x$ of degrees $n(q+1)$ and $n$,
  respectively;
\item the function field of $C_\alpha$ is $\kappa(C_\alpha) = \FF_q(T,x) = \FF_q(T,\alpha)$;
\item the rational function $T$ vanishes on $C_\alpha(\FF_q)$
\end{itemize}
In fact, these properties characterize the curves that arise from minimum-height
functions.

\begin{theorem}
  \label{thm:geom_to_arith}
  Suppose we are given a positive integer $n \geq 1$, a smooth, proper,
  geometrically irreducible curve $X_{/\FF_q}$, and separable rational functions
  $u,v \in \kappa(X)$ satisfying the following conditions:
  \begin{enumerate}
\item $\#X(\FF_q) = n(q+1)$;    
\item $u$ has degree $n(q+1)$ and $v$ has degree $n$;
\item $\kappa(X) = \FF_q(u,v)$; and
\item $u$ vanishes at all points of $X(\FF_q)$.
\end{enumerate}
Then $v$ generates a totally $u$-adic extension of $\FF_q(u)$ of degree
$n(q+1)$, and $h(v) = 1/(q+1)$.
\end{theorem}

\begin{proof}
  Write $\kappa(X)$ as an extension of $\FF_q(u)$. As $\kappa(X) = \FF_q(u,v)$
  has transcendence degree~1, there is an irreducible polynomial $f \in
  \FF_q(u)[x]$ such that $f(v) = 0$. By clearing denominators, we may assume
  that $f \in \FF_q[u][x]$. Viewing $u,x$ as indeterminates, this equation gives
  a (possibly singular) plane model for $X$.

  To see that $v$ is totally $u$-adic over $\FF_q(u)$, observe that $\kappa(X) /
  \FF_q(u)$ has degree $n(q+1)$, so $\kappa(X)$ has at most $n(q+1)$ places
  above $\ord_0$. By assumption, we see that $u$ annihilates all $n(q+1)$
  $\FF_q$-rational points of $X$, so $\kappa(X)$ has precisely $n(q+1)$ distinct
  places above $\ord_0$ (corresponding to the valuations given by vanishing at
  the points of $X(\FF_q)$).

  To compute the height of $v$, we must ascertain the degree of $f$ as a
  polynomial in $u$ and $x$. As $u$ has degree $n(q+1)$ as a rational function,
  fixing a generic value for $u$ gives $n(q+1)$ distinct $\bar \FF_q$-rational
  points on $X$. Thus $\deg_x(f) = n(q+1)$. Similarly, since $v$ has degree $n$
  as a rational function, fixing a generic value for $v$ gives $n$ $\bar
  \FF_q$-rational points on $X$. Hence $\deg_u(f) = n$. We conclude that
  \[
  h(v) = \frac{\deg_u(f)}{\deg_x(f)} = \frac{1}{q+1}. \qedhere
  \]
\end{proof}

Having shown that totally $T$-adic functions of minimum height are closely
related to certain special algebraic curves, we now give bounds on the genera of
those curves.

\begin{theorem}[Genus Inequalities]
  \label{thm:genus_bound}
  Let $C_{/\FF_q}$ be the algebraic curve associated to a totally $T$-adic function of
  height $\frac{1}{q+1}$. Write $g(C)$ for its genus and $n$ for its gonality. Then
  \[
  \frac{(n-1)(q+1)}{2\sqrt{q}} \leq g(C) \leq
  \frac{1}{2}(q+1)(n-1)^2 + \frac{1}{2}(q-1)(n-1).
  \]
\end{theorem}

\begin{proof}
  For ease of notation in the proof, write $g = g(C)$. The Weil bound shows that
  \[
\big| \# C(\FF_q) - (q+1) \big| \leq 2g\sqrt{q}.
\]
Using the count in Theorem~\ref{thm:point_count}, we find
\[
(n-1)(q+1) \leq 2g \sqrt{q},
\]
which is equivalent to the lower bound.

For the upper bound, we use the adjunction formula for a particular model of $C$
lying on the surface $\PP^1 \times \PP^1$. As above, we write $T$ and $x$ for
rational functions on $C$. To avoid confusion, we write $\pi_T$ and $\pi_x$ for
the corresponding coordinate projections $\PP^1 \times \PP^1 \to
\PP^1$. Combining these gives a commutative diagram
\[
\begin{tikzcd}
C \ar[rrd,bend left=30, "x"] \ar[ddr,bend right=30, "T"] \ar[dr,"\iota"] & & \\
& \PP^1 \times \PP^1 \ar[r,"\pi_x"] \ar[d, "\pi_T"]  & \PP^1 \ar[d] \\
&  \PP^1 \ar[r] & \Spec \FF_q
\end{tikzcd}
\]
Let $D$ be the image of the morphism $\iota$. Since $\kappa(C) = \FF_q(T,x)$,
the induced morphism $C \to D$ is birational. Note that $D$ has bidegree
$(n,n(q+1))$. 

Write $K$ for a canonical divisor on the surface $\PP^1 \times \PP^1$; it has
bidegree $(-2,-2)$. If $Y \subset \PP^1 \times \PP^1$ is a curve with bidegree
$(d_1,d_2)$, then the adjunction formula \cite[Prop.~V.1.5]{Hartshorne_Bible}
shows that $Y$ has arithmetic genus
\begin{align*}
  p_a(Y) &= 1 + \frac{1}{2}\left(Y^2 + K\cdot Y\right) \\
  &= 1 + \frac{1}{2}\left(2d_1d_2 - 2d_1 - 2d_2\right) \\
  &= (d_1 - 1)(d_2 - 1).
\end{align*}
Applying this formula to the curve $D$, we find that
\[
  p_a(D) = (n-1)\left(n(q+1) - 1\right) = (q+1)(n-1)^2 + q(n-1).
  \]
This gives a coarse upper bound on the (geometric) genus of $C$, but we can do
better if we estimate the singularities of $D$.

By abuse of notation, let us use $T,x$ as affine coordinates on $\PP^1 \times
\PP^1$. This abuse is justified as follows: if $f \in \FF_q[T][x]$ is the
minimal polynomial of our function with height $1/(q+1)$, then $f$ is an affine
equation for $D$.

We claim that $D$ passes through each $\FF_q$-rational point of the fiber
$\pi_T^{-1}(0)$ with multiplicity $n$. This is equivalent to saying that $f \in
\mm_P^n \smallsetminus \mm_P^{n+1}$ for $P \in \pi_T^{-1}(0)$, where $\mm_P$ is
the maximal ideal of the local ring of $P \in \PP^1 \times \PP^1$. Consider
first the point $(T,x) = (0,0)$ on the affine piece of
$D$. Proposition~\ref{prop:newton_criterion} describes the Newton polygon of $f$
as a polynomial in $x$. Its first segment has projection length $r$ and slope
$-1$, from which we deduce that each monomial $a_i(T)x^i$ of $f$ has $i +
\ord_0(a_i) \geq r$. So $f \in \mm^r_{(0,0)}$. Moreover, the constant term is of
the form $aT^r$ for some nonzero $a \in \FF_q$, so $f \not\in
\mm^{r+1}_{(0,0)}$. We have a similar statement for the Newton polygon of
$f(x+u)$ for $u \in \FF_q$, so $f \in \mm^r_{(0,u)} \smallsetminus
\mm^{r+1}_{(0,u)}$. Finally, replacing $f$ with its reciprocal polynomial allows
us to look at the point $(0,\infty)$. Since the Newton polygon for $f$ is
left-to-right symmetric, its Newton polygon is unaffected by passing to the
reciprocal polynomial. Hence, $f \in \mm^r_{(0,\infty)} \smallsetminus
\mm^{r+1}_{(0,\infty)}$. Now we invoke Corollary~\ref{cor:r=n} to see that $r =
n$.

Write $\mu_P$ for the multiplicity to which $D$ passes through the point $P$ of
$\PP^1 \times \PP^1$. Section V.3 of \cite{Hartshorne_Bible} and
the preceding paragraph shows that the geometric genus of $D$ is given by
\begin{align*}
  g &\le p_a(D) - \sum_{P \in D} \frac{1}{2}\mu_P (\mu_P - 1) \\
  &\leq p_a(D) - \sum_{\substack{P \in D(\FF_q) \\ \pi_T(P) = 0}} \frac{1}{2}n(n-1) \\
  &= (q+1)(n-1)^2 + q(n-1) - (q+1)\cdot \frac{1}{2}n(n-1) \\
  &= \frac{1}{2}(q+1)(n-1)^2 + \frac{1}{2}(q-1)(n-1).
\end{align*}
This completes the proof of the upper bound.
\end{proof}


\section{Constructions: examples of minimum height}
\label{sec:examples}

Our goal in this section is to address two questions:

\begin{question}
  For a given finite field $\FF_q$, does there exist a function $\alpha \in
  \cT_q \smallsetminus \FF_q$ with minimum height: 
  $h(\alpha) = \frac{1}{q+1}$? (Yes.)
\end{question}

\begin{question}
  Given a finite field $\FF_q$ and a positive integer $n$, does there exist a
  function $\alpha \in \cT_q \smallsetminus \FF_q$ with gonality~$n$ such that
  $h(\alpha) = \frac{1}{q+1}$? (In many cases, yes.)
\end{question}

The second question is a refinement of the first. We completely answer the first
question with an elementary construction. The associated algebraic curves are
all rational, so no interesting geometry is involved. To answer the second
question, we give geometric constructions. We use cyclic $n$-covers of the
projective line to produce examples when $n \mid q - 1$, and we use
hyperelliptic curves to produce examples with $n = 2$ when $q > 2$ is even. We
remark that none of these constructions yields an example of $\alpha \in \cT_q$
with height $\frac{1}{q+1}$ and gonality $n \geq q$.

Given a nonconstant polynomial $f \in \FF_q[T][x]$, it is natural to define its
height to be $h(f) = \frac{\deg_T(f)}{\deg_x(f)}$.  If $f$ is irreducible and
$\alpha$ is a root of $f$ in $\overline{\FF_q(T)}$, then $h(f) = h(\alpha)$. The
averaging argument in the next lemma shows that the height of $f$ is at least
the minimum of the heights of its irreducible factors. We will use this to
provide an easy irreducibility criterion in our first construction of functions
of minimum height.

\begin{lemma}
  \label{lem:weighted_average}
Let $f(x)\in\FF_q[T][x]$ be a polynomial with no root in $\FF_q$ and which
splits completely over $\FF_q\Ls{T}$. Assume that
\[
(q+1)\deg_T(f)=\deg_x(f).
\]  
Then each root $\alpha$ of $f(x)$ has height $h(\alpha)=\frac{1}{q+1}$ and
degree $[\FF_q(T)(\alpha):\FF_q(T)]$ equal to a multiple of $q+1$.
\end{lemma}

\begin{proof}
We may assume that $f$ is nonconstant as a polynomial in $x$.  Factor
$f(x)=f_1(x)\cdots f_m(x)$ into irreducible polynomials in $\FF_q[T][x]$, and
let $\alpha_i$ be a root of $f_i(x)$. By hypothesis, each $\alpha_i$ is totally
$T$-adic and not in $\FF_q$. We find that
\begin{equation}
  \label{eq:WeightedAverage}
\frac{1}{q+1}=\frac{\deg_T(f)}{\deg_x(f)}=r_1h(\alpha_1)+r_2h(\alpha_2)+\dots +r_mh(\alpha_m),
\end{equation}
where $r_i=\deg_x(f_i)/\deg_x(f)$.  Since $\sum r_i=1$,
\eqref{eq:WeightedAverage} is a weighted average of the heights $h(\alpha_i)$,
which are all at least $\frac{1}{q+1}$ by Theorem~\ref{thm:height_bound}.  We
conclude that $h(\alpha_i)=\frac{1}{q+1}$ for all $i$.  That each degree
$[\FF_q(T)(\alpha_i):\FF_q(T)]$ is a multiple of $q+1$ follows immediately from
Proposition~\ref{prop:newton_criterion}.
\end{proof}

\begin{theorem}
  \label{thm:min_height}
  Fix a prime power $q$. There exists $\alpha \in \cT_q$ of height
  $\frac{1}{q+1}$ and gonality~1.
\end{theorem}

\begin{proof}
Suppose first that $q$ is odd and let $c \in \FF_q$ be a nonsquare. Let $\alpha$
be a root of
\[
f(x) = Tx^{q+1} + x^q - x - cT.
\]
The Newton polygon of $f$ with respect to $\ord_0$ shows that it has one root
$\beta$ with $\ord_0(\beta) = -1$. Each $t \in \FF_q$ is a simple root of $f
\pmod{T}$, so an application of Hensel's Lemma shows that $f$ splits completely
over $\FF_q\Ls{T}$. Note that $f$ is nonvanishing on $\FF_q$ because $c$ is not
a square. Hence, $f$ satisfies the hypotheses of
Lemma~\ref{lem:weighted_average}. We deduce that $f$ is irreducible and $\alpha$
has height $\frac{1}{q+1}$.


To complete the proof, we need to consider the case of $q$ even. Choose $c \in
\FF_q$ such that $x^2 + x + c$ has no $\FF_q$-rational root. This is possible
because the Artin-Schreier map $x \mapsto x^2 + x$ is an additive homomorphism
with kernel $\FF_2$; in particular, it is not surjective. Let $c$ be an element
of $\FF_q$ that is not in the image. Let $\alpha$ be a root of
\[
f(x) = Tx^{q+1} + x^q + (T+1)x + cT.
\]
The proof now proceeds exactly as in the case of $q$ odd.   
\end{proof}

\begin{theorem}
  \label{thm:qr_arithmetic}
Let $n \geq 1$ be an integer, and let $q \equiv 1 \pmod n$ be a prime
power. Let $\alpha$ be a root of the polynomial
\[
T^n x^n (x^q - x)^n - (x^q - x)^n + T^n \in \FF_q(T)[x]
\]
Then $\alpha$ is totally $T$-adic with height $1/(q+1)$ and gonality $n$.
\end{theorem}

Presumably, we could argue that the polynomial in the theorem is totally
$T$-adic using Hensel's lemma, and there is a trick for
showing that it is irreducible. But, at the risk of exposing the conceit, we
allow that we discovered these examples using a simple family of cyclic covers
of the projective line, and our algebraic curve machinery gives a conceptual
geometric proof.

\begin{lemma}
Let $n \geq 1$ be an integer and $q \equiv 1 \pmod n$ a prime power. Consider
the cyclic cover of $\PP^1$ defined by the following affine equation:
\begin{equation}
  \label{eq:cyclic_ex}
  X_{/\FF_q} \colon y^n = x^n(x^q - x)^n + 1.
\end{equation}
Then $X$ is geometrically irreducible.
\end{lemma}

\begin{proof}
  It suffices to show that
  \[
  y^n - x^n(x^q - x)^n - 1
  \]
  is irreducible in $\bar \FF_q[x,y] = \bar \FF_q[x][y]$. By Gauss's lemma, we further
  reduce to showing that it is irreducible as a univariate polynomial in $y$
  over the field $\bar \FF_q(x)$. To that end, we will use the well known criterion that
  $y^n - c \in k[y]$ is irreducible over a field $k$ if and only if $c \not\in k^\ell$
  for every prime $\ell \mid n$ and $c \not\in -4k^4$ if $4 \mid n$
  \cite[VI,\S9]{Lang_Algebra_2005}.

  Suppose that $\ell \mid n$ for some prime $\ell$. Set $u = x(x^q - x)$. If we
  assume for the sake of a contradiction that $u^n + 1 \in \bar \FF_q(x)^\ell$, then
  there is $v \in \bar \FF_q[x]$ such that
  \[
  u^n + 1 = v^\ell.
  \]
  This equation immediately implies that $u,v$ are coprime. Taking the derivative of both sides, we find that
  \[
  n\ u'\ u^{n-1} = \ell\ v' v^{\ell - 1}.
  \]
  Since $q \equiv 1 \pmod n$, we find that $n$ and $\ell$ are nonzero in
  $\bar \FF_q$. It follows that $u^{n-1} \mid v'$. This implies an inequality of degrees:
  \[
  (n-1)(q+1) = \deg\left(u^{n-1}\right) \leq \deg(v') \le \deg(v) - 1
  = \frac{n}{\ell}(q+1) - 1.
  \]
  Dividing both sides by $q+1$ shows that
  \[
  n - 1 \leq \frac{n}{\ell} - \frac{1}{q+1} < \frac{n}{\ell}.
  \]
  This is absurd unless $n = 1$. But we also assumed $\ell$ is a prime dividing
  $n$, so this is a contradiction.

  The argument in the preceding paragraph goes through essentially verbatim if
  we assume that $4 \mid n$ and $u^n + 1 = -4 v^4$. Note that $n$ even implies
  $q$ is odd, so $4 \ne 0$ in $\bar \FF_q$. We conclude that $y^n - (u^n + 1)$ is
  irreducible over $\bar \FF_q(x)$, and hence $X$ is geometrically irreducible.
\end{proof}

\begin{proof}[Proof of Theorem~\ref{thm:qr_arithmetic}]
Let $\tilde X$ be the normalization of the curve $X$ in \eqref{eq:cyclic_ex}. We
will apply Theorem~\ref{thm:geom_to_arith} to $\tilde X$ after showing that
it satisfies all the requisite properties.  We begin by looking at related
properties of $X$.

First, we show that $X$ has no $\FF_q$-rational singularity.  Setting $z =
y/x^{q+1}$ and $w = 1/x$ gives an equation for $X$ near infinity:
\[
z^n = (1 - w^{q-1})^n + w^{nq+n}.
\]
Since $n \mid (q-1)$, we find that $\mu_n \subset \FF_q$. Thus all points of $X$
at infinity are $\FF_q$-rational, of the form $(0,\zeta)$ for $\zeta \in \mu_n$.
The Jacobian criterion shows that the model $z^n = (1-w^{q-1})^n + w^{nq+n}$ is
nonsingular at all such points.  For the affine patch, we set $g(x) =
x^n(x^q-x)^n + 1$. By the Jacobian criterion, $X$ is singular at a point
$(\rho,\theta)$ if and only if $g(\rho) = g'(\rho) = 0$. For any $u \in
\FF_q$, we see that $g(u) = 1$, so $X$ has no $\FF_q$-rational singular point.

Next, we show that $\#\tilde X(\FF_q) = n(q+1)$. As $\tilde X$ is an $n$-fold
cover of $\PP^1$, we see that $\#\tilde X(\FF_q) \leq n(q+1)$. For the opposite
inequality, observe that $(t,\zeta)$ is an $\FF_q$-rational point of $X$ for $t
\in \FF_q$ and $\zeta \in \mu_n$. In the previous paragraph, we saw that the
model for $X$ near infinity has $n$ additional rational points. No singularity
of $X$ is $\FF_q$-rational, so we conclude that $\# \tilde X(\FF_q) = \#X(\FF_q)
\geq n(q+1)$.

Finally, we observe that $T := (x^q - x) / y$ is a rational function of
degree~$n(q+1)$ which annihilates $X(\FF_q)$. Since $T$ is also a rational
function on $\tilde X$, and since
\[
\kappa(\tilde X) = \kappa(X) = \FF_q(y,x)
= \FF_q(T,x),
\]
we have proved all of the required properties of $\tilde X$. The conclusion of
Theorem~\ref{thm:geom_to_arith} applies with $u = T$ and $v = x$. To get
the explicit form for the polynomial defining $\alpha$, we replace $y$ with
$(x^q - x)/T$ in \eqref{eq:cyclic_ex} and clear denominators. 
\end{proof}

Though we did not need the curve $X$ from \eqref{eq:cyclic_ex} to be smooth for
our arithmetic application, it is amusing to note that $X$ actually is smooth if
we impose a mild additional hypothesis on the characteristic of $\FF_q$.

\begin{proposition}
Let $n \geq 1$ be an integer and $q \equiv 1 \pmod n$ a prime power. Consider
the cyclic cover of $\PP^1$ defined by the affine equation
\[
  X_{/\FF_q} \colon y^n = x^n(x^q - x)^n + 1.
  \]
Then $X$ is nonsingular if and only if $\gcd(x^{2n}+1,x^{q-1} - 2) = 1$. In
particular, $X$ is nonsingular if $q$ is not a power of $3$ or $5$.
\end{proposition}

\begin{proof}
The equation $z^n = (1-w^{q-1})^n + w^{nq+n}$ is nonsingular at all points $(w,z) =
(0,\zeta)$ for $\zeta \in \mu_n$, so $X$ is nonsingular at infinity.

To deal with the standard affine patch, we set $g(x) = x^n(x^q-x)^n + 1$. Suppose that $X$ is singular at a point $(\rho, \theta)$. The Jacobian
criterion shows that this is equivalent to
$g(\rho) = g'(\rho) = 0$. Now
\[
g'(x) = nx^n(x^q - x)^{n-1}(x^{q-1} - 2).
\]
The factors $x(x^q-x)$  vanish precisely along $\FF_q$, but $g(x)$ does not
vanish at any such element. So a common solution must satisfy $\rho^{q-1} =
2$. Inserting this into $g(x)$, we find that
\[
0 = g(\rho) = \rho^n\left(\rho^q - \rho\right)^n + 1 = \rho^{2n} + 1.
\]
Hence, $\rho^{2n} = -1$.  These calculations are reversible, so we see that $X$
is nonsingular if and only if $\gcd(x^{2n} + 1, x^{q-1} - 2) = 1$.

Suppose now that $X$ is singular, and let $\rho \in \bar \FF_q$ satisfy
$\rho^{2n} = -1$ and $\rho^{q-1} = 2$. As $q \equiv 1 \pmod n$, we observe that
  $q$ is congruent to $1$ or $n+1$ modulo~$2n$. In the former case, we write $q
  = 1 + 2nk$, and find that
\[
2 = \rho^{q-1} = (\rho^{2n})^k = \pm 1.
\]
This can only happen in characteristic~3. If instead $q = n + 1 + 2nk$, then
\[
2 = \rho^{q-1} = (\rho^{2n})^k \cdot \rho^n = \pm \rho^n.
\]
Squaring both side shows that $-1 = \rho^{2n} = 4$. Hence $5 \mid q$. We
conclude that $X$ is nonsingular if $q$ is not a power of $3$ or $5$.
\end{proof}

\begin{remark}
  Whenever the curve $X$ in \eqref{eq:cyclic_ex} is smooth, we can use the
  Hurwitz formula to conclude that it has genus
\[
g(X) = \frac{1}{2}(q+1)(n-1)^2 + \frac{1}{2}(q-1)(n-1),
\]
which is precisely the upper bound given by
Theorem~\ref{thm:genus_bound}. Examples with $q = 3$ and $n = 3$ will be given
in the next section to show that this upper bound is not always attained.
\end{remark}

\begin{theorem}
  \label{thm:char2}
  Fix $s > 1$ and set $q = 2^s$. There exists $\alpha \in \cT_q$ of height
  $\frac{1}{q+1}$ and gonality~2.
\end{theorem}

\begin{proof}
  Choose $\zeta \in \FF_q \smallsetminus \FF_2$, and let $Q \in \FF_q[x]$ be a monic
  irreducible polynomial of degree $q+1$. Set
  \[
  P(x) = x(x^q+x)Q(x). 
  \]
  We define a hyperelliptic curve with affine equation
  \[
    X_{/\FF_q} : y^2 + (1+\zeta) Q(x)y = \zeta P(x).
  \]
  An equation for $X$ near infinity is given by setting $z = y/x^{q+1}$ and $w = 1/x$:
  \[
    z^2 + (1+\zeta) Q^*(w)z = \zeta P^*(w),
    \]
  where $Q^*(w) = w^{\deg(Q)}Q(1/w)$ is the reversed polynomial for $Q$, and
  similarly for $P$. One verifies readily that $X$ is nonsingular of
  genus~$q$ \cite[Rem.~7.4.25]{Qing_Liu_Algebraic_Geometry}.

  Since $P(u) = 0$ for all $u \in \FF_q$ and $Q(u) \ne 0$ for $u \in \FF_q$, we
  see immediately that $X$ admits $2q$ affine rational points. At infinity, we
  observe that $P^*(0) = Q^*(0) = 1$, so that the two points at infinity are
  $(0,\theta)$, where $\theta$ satisfies the equation $z^2 + (1+\zeta)z + \zeta
  = 0$. Evidently, the solutions to this are $\theta = 1, \zeta$; write
  $\infty^+$ and $\infty^-$ for these two points, respectively.  Thus, we find
  that $\#X(\FF_q) = 2(q+1)$.

  In order to invoke Theorem~\ref{thm:geom_to_arith} (with $n = 2$), we must
  produce a function $T$ such that $X$ has function field $\FF_q(x,T) =
  \FF_q(x,y)$ and the divisor of zeros of $T$ is precisely $X(\FF_q)$. We claim that
  \[
    T = \frac{(x^q+x)(y+\zeta)}{\zeta + (1+\zeta) Q(x) + P(x)}
    \]
    is such a function. Write $R(x) = \zeta + (1+\zeta) Q(x) + P(x)$. One finds that
  \begin{eqnarray*}
    \div(x^q + x) &=& \sum_{u \in \FF_q} (u,0) + \big(u,(1+\zeta)Q(u)\big) - q\ \infty^+ - q\ \infty^- \\
    \div(y+\zeta) &=& \sum_{R(u) = 0} (u,\zeta) - (q+1)\ \infty^+ - (q+1)\ \infty^- \\
    \div(R(x)) &=& \sum_{R(u) = 0} (u,\zeta) + \big( u, \zeta+(1+\zeta)Q(u)\big) - (2q+2)\ \infty^+ - (2q+2)\ \infty^-.
  \end{eqnarray*}
  Combining these gives
  \[
  \div(T) = \sum_{p \in X(\FF_q)} p - \sum_{R(u) = 0} \big( u, \zeta+(1+\zeta)Q(u)\big).
  \qedhere
  \]  
\end{proof}

\begin{remark}
  The curve constructed in the proof of Theorem~\ref{thm:char2} has genus~$q$,
  which agrees with the upper bound given by Theorem~\ref{thm:genus_bound} in the case $n = 2$. 
\end{remark}

Theorem~\ref{thm:qr_arithmetic} applies with $n = 2$ for any odd prime power $q$, and Theorem~\ref{thm:char2} takes care of gonality $2$ when $q$ is even:

\begin{corollary}
Let $q > 2$ be a prime power. There exists $\alpha \in \cT_q$ of height
  $\frac{1}{q+1}$ and gonality~2.
\end{corollary}

In Section~\ref{sec:non-existence}, we will give a computational proof that the
corollary does not extend to the case $q = 2$.


\section{Calculations: examples for $q = 3$ and $n = 3$}
\label{sec:naive_algorithm}

In this section, we leverage the description of minimal polynomials of totally
$T$-adic functions of height $\frac{1}{q+1}$ in
Proposition~\ref{prop:newton_criterion} to give a naive algorithm for finding
examples.  It is not particularly efficient because it randomly searches an
enormous haystack with a small number of needles. However, it does succeed at
finding examples when $q = 3$ and $n = 3$, which none of the constructions in
the previous section could do.

\begin{algorithm}[ht]
  \begin{flushleft}
    \textbf{Input.} a prime power $q$, a  positive integer $n$ (gonality),
    and a  positive integer $B$ (search bound)

    \medskip
    
    \textbf{Output.} an irreducible polynomial in $\FF_q[T][x]$ that splits
    completely over $\FF_q\Ps{T}$ and whose roots have height $\frac{1}{q+1}$,
    or None if no such polynomial is located after $B$ tries

    \bigskip

    For $j = 1, \ldots, B$:

    \medskip

    \begin{enumerate}
    \item For $i = 1, \ldots, n$ and $u \in \FF_q$, select $a_i, b_{u,i} \in
      \FF_q[T]$ uniformly at random with nonzero constant coefficient such that
      $\deg(a_i) \leq n$ and $\deg(b_{u,i}) \leq n-1$.
    
    \item Construct the product
    \[
    f_0(x) = \prod_{i=1}^n \left(Tx - a_i\right)
    \prod_{u \in \FF_q} \prod_{i=1}^n \left(x - u - Tb_{u,i}\right).
    \]
    Write $f_0(x) = f(x) + T^{n+1} \ g(x)$ for some polynomials $f,g \in \FF_q[T,x]$,
    where $\deg_T(f) \leq n$.
    \item For each $u \in \FF_q$, check if the Newton polygon of $f(x+u)$ with
      respect to $\ord_0$ has vertices $(0,n)$, $(n,0)$, $(nq,0)$, and
      $(n(q+1),n)$. If not, continue.
    \item If $f$ is not irreducible, continue.
    \item If $f$ splits completely over $\FF_q\Ps{T}$, return $f$.
    \end{enumerate}
    \medskip
    Return None.
  \end{flushleft}
\caption{--- Randomly search for elements of $\cT_q$ of minimum positive height and gonality $n$}
\label{alg:Fq}
\end{algorithm}

The first two steps of Algorithm~\ref{alg:Fq} construct a polynomial $f \in
\FF_q[T][x]$ with leading term $T^nx^{n(q+1)}$ and with $\deg_T(f) = n$. The
tests in Step~(3) are very fast, and Proposition~\ref{prop:newton_criterion} and
Corollary~\ref{cor:r=n} show that they must be satisfied by the minimal
polynomial of a totally $T$-adic function of height~$\frac{1}{q+1}$ and
gonality~$n$. Evidently, Steps~(4) and~(5) ensure that the output of
Algorithm~\ref{alg:Fq} is correct. To test if a given $f \in \FF_q[T][x]$ is
irreducible, one can use Gr\"obner basis techniques to compute the primary
decomposition of the ideal $(f) \subset \FF_q[T,x]$.  This is implemented in
Singular, and consequently in Sage \cite{sagemath.9.1}. To determine if $f$
splits completely over $\FF_q{\Ps{T}}$, one can use MacLane's algorithm
\cite{MacLane2} to decompose the ideal $T\FF_q[T]$ in the extension $\FF_q(T)[x]
/ (f)$, which is also implemented in Sage.

We ran Algorithm~\ref{alg:Fq} with $q = 3$, $n = 3$, and $B = 2^{20}$ and
discovered the following example in approximately 20 minutes:
{\small
\begin{align*}
&T^{3} x^{12} + 2 T^{2} x^{11} + \left(2 T^{3} + 2 T\right) x^{10} + \left(T^{2}
  + 1\right) x^{9} + \left(T^{2} + T\right) x^{8} + \left(T^{3} + 2 T^{2}\right)
  x^{7} \\ 
  & + \left(2 T^{3} + 2 T\right) x^{6} + 2 T^{3} x^{5} + \left(2 T^{3} + 2\right) x^{3} + \left(2 T^{2} + T\right)
  x^{2} + \left(T^{3} + T^{2}\right) x + 2 T^{3}.
\end{align*}
} The smooth projective algebraic curve with this plane model has
genus~6.

As the algorithm is non-deterministic, other runs required anywhere from a few
seconds to upwards of an hour to find an example. After a number of additional
runs of Algorithm~\ref{alg:Fq} with $q = 3$ and $n = 3$, we found a more compact
example:
{\small
\[
  T^3 x^{12} + T^2 x^{11} + (2 T^3 + 2 T) x^{10} + 2x^9 + (T^3 + 2T)x^8 + 2Tx^6 + (2T^2 + 1)x^3 + T^3x^2 + 2T^3
\]
}
The smooth projective algebraic curve with this plane model has genus~4.

\begin{remark}
Theorem~\ref{thm:genus_bound} shows that the curve associated to a
minimum-height totally $T$-adic function over $\FF_3(T)$ with gonality~3 has
genus $3 \leq g \leq 10$.
\end{remark}


\section{More Calculations: no example for $q = 2$ and $n = 2,3,4$}
\label{sec:non-existence}

In this section, we give an algorithm for locating \textit{all} functions of
minimum height in $\cT_2$ of fixed gonality $n$. It fails to find any when $n =
2, 3, 4$, and it shows that the proof of Theorem~\ref{thm:min_height}
essentially identified all examples when $n = 1$. Before describing the
algorithm in full, we illustrate it in the case $n = 2$.

\begin{theorem}
  \label{thm:no_q=2,n=2}
  There does not exist a totally $T$-adic function over $\FF_2(T)$ with height
  $1/3$ and gonality $2$.
\end{theorem}

\begin{proof}
 Assume for the sake of a contradiction that there is a function of height $1/3$
 and gonality~2. We let $f \in \FF_2[T][x]$ be its minimal polynomial. Then the
 following properties hold for $f$:
 \begin{enumerate}
    \item $\deg_x(f) = 6$ and $\deg_T(f) = 2$;
    \item The leading and constant coefficients of $f$ are $T^2$;   
    \item The Newton polygon of $f(x)$ with respect to
    $\ord_0$ has vertices $(0,2)$, $(2,0)$, $(4,0)$, and $(6,2)$;
    \item $\ord_0\left(f(1/T)\right) \geq -2$; and
    \item $\ord_0\left(f(T+u)\right) \geq 4$ for $u = 0,1$.
\end{enumerate}
The first, second, and third properties are an immediate consequence of
Proposition~\ref{prop:newton_criterion}.  The fourth and fifth properties follow
upon evaluating $f$ in factored form using the information in
Corollary~\ref{cor:well-distributed}:
\[
f(x) = T^2 \prod_{i=1}^2 \left(x + \frac{1}{T} + b_i \right)
\prod_{i=3}^4 
\left(x + 1 + T + T^2b_i\right)
\prod_{i=5}^6
\left(x + T + T^2b_i \right)
\]
where $b_i \in \FF_2\Ps{T}$. 

Properties (1) and (2) show that $f$ must have the form
\[
f(x) = T^2 x^6 + A_5(T) x^5 + A_4(T)x^4 + A_3(T)x^3 + A_2(T)x^2 + A_1(T)x + T^2,
\]
where $A_i(T) = a_{i,2}T^2 + a_{i,1}T + a_{i,0}$ for $i = 1, \ldots, 5$. The
remaining properties of $f$ yield linear conditions on the coefficients
$a_{i,j}$. Property (3) implies that
\begin{eqnarray*}
  a_{5,0} &=& 0 \\
  a_{4,0} &=& 1 \\
  a_{2,0} &=& 1 \\
  a_{1,0} &=& 0.
\end{eqnarray*}
Evaluating at $x = 1/T$ gives
\[
f(1/T) = a_{5,0}\frac{1}{T^5} + (1 + a_{4,0} + a_{5,1})\frac{1}{T^4}
+ (a_{3,0} + a_{4,1} + a_{5,2})\frac{1}{T^3} + (a_{2,0} + a_{3,1} + a_{4,2})\frac{1}{T^2} + \cdots
\]
Property (4) gives the linear conditions
\begin{eqnarray*}
  a_{5,0} &=& 0 \\
  a_{4,0} + a_{5,1} &=& 1 \\
  a_{3,0} + a_{4,1} + a_{5,2} &=& 0.
\end{eqnarray*}
Evaluating at $x = T$ gives
\[
f(T) = a_{1,0}T + (1 + a_{1,1} + a_{2,0})T^2 + (a_{1,2} + a_{2,1} + a_{3,0})T^3
+ (a_{2,2} + a_{3,1} + a_{4,0})T^4 + \cdots 
\]
Property (5) gives the conditions
\begin{eqnarray*}
  a_{1,0} &=& 0 \\
  a_{1,1} + a_{2,0} &=& 1 \\
  a_{1,2} + a_{2,1} + a_{3,0} &=& 0.
\end{eqnarray*}
Similarly, evaluating at $x = T+1$ gives the conditions
\begin{eqnarray*}
  a_{1,0} + a_{2,0} + a_{3,0} + a_{4,0} + a_{5,0} &=& 0 \\
  a_{1,0} + a_{1,1} + a_{2,1} + a_{3,0} + a_{3,1} + a_{4,1} + a_{5,0} + a_{5,1} &=& 0 \\
  a_{1,1} + a_{1,2} + a_{2,0} + a_{2,2} + a_{3,0} + a_{3,1} + a_{3,2} + a_{4,2} + a_{5,1} + a_{5,2} &=& 0 \\
  a_{1,2} + a_{2,1} + a_{3,0} + a_{3,1} + a_{3,2} + a_{5,2} &=& 0
\end{eqnarray*}

Solving this system of 14 inhomogeneous equations gives a 3-dimensional space of
solutions, and consequently $8 = 2^3$ possibilities for $f$. Four of them have
the property that $f(1) = 0$, which violates the irreducibility of $f$. The
remaining polynomials are
\begin{eqnarray*}
& T^{2} x^{6} + \left(T^{2} + 1\right) x^{4} + x^{2} + T^{2} \\
& T^{2} x^{6} + x^{4} + \left(T^{2} + 1\right) x^{2} + T^{2}\\
& T^{2} x^{6} + T^{2} x^{5} + \left(T^{2} + T + 1\right) x^{4} + T^{2} x^{3} + \left(T^{2} + T + 1\right) x^{2} + T^{2} x + T^{2} \\
& T^{2} x^{6} + T^{2} x^{5} + \left(T + 1\right) x^{4} + T^{2} x^{3} + \left(T + 1\right) x^{2} + T^{2} x + T^{2}.
\end{eqnarray*}
The first two of these are perfect squares, while the third factors as
\[
\left(T x^3 + x^2 + (T+1)x + T\right)\left(Tx^3 + (T+1)x^2 + x + T\right).
\]
If $f$ is the fourth polynomial, then
\[
T^2f\left(\frac{1+Tx}{T}\right) \pmod T = x^2 + x + 1,
\]
so that $f$ cannot factor completely over $\FF_2\Ps{T}$. (Alternatively, the
fourth polynomial has three unramified quadratic places above $T$, as can be
verified, for example, by the method of MacLane \cite{MacLane2}).

In summary, we have shown that no polynomial has the requisite properties to be
the minimal polynomial of a totally $T$-adic function with height $1/3$ and
gonality~2.
\end{proof}

The strategy of the above proof generalizes immediately:

\begin{algorithm}[ht]
  \begin{flushleft}
    \textbf{Input.} a positive integer $n$

    \medskip
    
    \textbf{Output.} a complete list of polynomials $f \in \FF_2[T][x]$ such
    that $f$ is the minimal polynomial of a totally $T$-adic function with
    height~$\frac{1}{3}$ and gonality~2

    \bigskip

    Initialize an empty list $L$. \\

    \medskip

    Define polynomials $A_i(T) = a_{i,0} + a_{i,1}T + \cdots + a_{i,n}T^n$ for $i = 1, \ldots, 3n-1$ with generic coefficients.  Set
    \[
      \tilde f(x) = T^n x^{3n} + A_{3n-1} x^{3n-1} + A_{3n-2} x^{3n-2} + \cdots
      + A_1 x + T^n.
    \]

    \medskip

    Obtain a system of linear equations in the coefficients $a_{i,j}$ by
    imposing the following conditions:
    \begin{itemize}
    \item $T^{n-i} \mid A_i$ for $i = 1, \ldots, n-1$
    \item $T^{n-i} \mid A_{3n-i}$ for $i = 1, \ldots, n-1$
    \item $A_n(0) = 1$ and $A_{2n}(0) = 1$
    \item $\ord_0 f(1/T) \geq -n$
    \item $\ord_0 f(T) \geq 2n$ and $\ord_0 f(T+1) \geq 2n$
    \end{itemize}
    Solve this system to get a set $S$ of polynomials.

    \medskip

    For each $f \in S$:
    \begin{enumerate}
      \item If $f$ is not irreducible, continue.
      \item If $f$ splits completely over $\FF_2\Ps{T}$, append $f$ to $L$.
    \end{enumerate}

    \medskip

  \end{flushleft}
\caption{--- Find all elements of $\cT_2$ of minimum positive height and gonality $n$}
\label{alg:F2}
\end{algorithm}

\begin{theorem}
  Let $\alpha \in \overline{\FF_2(T)}$ be a root of the polynomial $Tx^3 + x^2 +
  (T+1)x + T$. Then the following six functions are the only elements of $\cT_2$
  with gonality~1 and height $\frac{1}{3}$:
  \[
  \alpha, \ \frac{1}{\alpha+1}, \ \frac{\alpha+1}{\alpha}, \ \frac{1}{\alpha}, \
  \alpha + 1, \ \frac{\alpha}{\alpha+1}.
  \]
\end{theorem}

\begin{proof}
  We walk through Algorithm~\ref{alg:F2} in the case $n = 1$. The Newton polygon
  restrictions on the polynomial $f$ show that it must have the form
  \[
  f(x) = Tx^3 + (1 + a_{2,1}T)x^2 + (1 + a_{1,1}T)x + T.
  \]
  A brief computation shows that the conditions $\ord_0 f(1/T) \geq -1$ and
  $\ord_0 f(T) \geq 2$ are superfluous, while the condition $\ord_0 f(T+1) \geq
  2$ implies that $1 + a_{1,1} + a_{2,1} = 0$. The two solutions are
  \[
  Tx^3 + x^2 + (T+1)x + T \qquad \text{and} \qquad
  Tx^3 + (T+1)x^2 + x + T.
  \]
  The proof of Theorem~\ref{thm:min_height} shows that $Tx^3 + x^2 + (T+1)x + T$
  is irreducible and splits completely over $\FF_2\Ps{T}$. Let $\alpha$ be a
  root of this polynomial. One easily verifies that $1/(\alpha+1)$ and
  $(\alpha+1) / \alpha$ are the other two roots, and that the three elements
  $1/\alpha$, $\alpha+1$, and $\alpha / (\alpha + 1)$ satisfy the polynomial
  $Tx^3 + (T+1)x^2 + x + T$.
\end{proof}

Let us now analyze the size of the search space in Algorithm~\ref{alg:F2}. The
total number of coefficients $a_{i,j}$ in Algorithm~\ref{alg:F2} is
$(3n-1)(n+1)$. The Newton polygon conditions impose $n(n-1) + 2$ independent
linear conditions on these coefficients. Having imposed these, any resulting
polynomial will have $n$ roots with valuation~1, $n$ roots with valuation~$-1$,
and $n$ roots with valuation~$0$.

We may assume that $n > 1$. The condition $\ord_0 f(1/T) \geq -n$ insures that
for any root $\alpha$ with valuation~$-1$, the coefficient on $T^{-1}$ in its
$T$-adic expansion is~1. This yields $n$ new independent linear
conditions. Similarly, the condition $\ord_0 f(T) \geq 2n$ gives $n$ independent
linear conditions. Finally, the condition $\ord_0 f(T+1) \geq 2n$ insures that
for any of the $n$ roots $\alpha$ with valuation~0, the $T$-adic expansion has
the form $1 + T + \cdots$, which amounts to $2n$ independent linear
conditions. In total, we find that the $\FF_2$-dimension of the affine search
space $S$ is
\[
  (3n-1)(n+1) - [n(n-1) + 2] - n - n - 2n = 2n^2 - n - 3.
\]

We implemented Algorithm~\ref{alg:F2} in Sage \cite{sagemath.9.1} and carried it
out for the cases $n = 2, 3, 4$. 

\begin{theorem}
  \label{thm:no_q=2,n=3,4}
  There does not exist a totally $T$-adic function over $\FF_2(T)$ with height
  $1/3$ and gonality $2$, $3$, or $4$.
\end{theorem}

We have already described in detail the case $n = 2$ above. The computation for
$n = 3$ requires one to check $2^{12}$ polynomials; it took approximately 3
minutes on a 2.6 GHz Intel Core i5 processor with 16GB of RAM.

The computation for $n=4$ requires one to look at $2^{25}$ polynomials. Naively,
one might expect this to take $2^{13} = 8192$ times longer than the $n = 3$
computation, or around 2 weeks. However, checking for irreducibility and total
splitting becomes more onerous as the degree increases, so it would actually
take significantly longer if we ran it on one processor. Instead, we distributed
the computation across $32$ CPUs --- Xeon(R) E5-2699 v3 \@ 2.30GHz with 500GB
memory --- each running its own Sage process and handling $2^{20}$
polynomials. In this way, we reduced the wall clock time to just under 1 day.

The computation for $n = 5$ seems completely infeasible.


\section{Dynamics: examples of small height}
\label{sec:small_height}

Our goal for this section is to prove the following result:

\begin{theorem}
  \label{thm:liminf}
  Let $q$ be a prime power. Then
  \[
  \frac{1}{q+1} \leq \liminf_{\alpha \in \cT_q} h(\alpha) \leq \frac{1}{q-1}.
  \]
\end{theorem}

The proof uses a construction inspired by \cite{Petsche-Stacy}. Throughout this
section, we work with the polynomial
\[
\phi(x) = \frac{x^q - x}{T}. 
\]
For $j \geq 1$, define
\[
f_j(x) = T^{1+q+\cdots+q^{j-1}} \left[\phi^j(x) - 1\right],
\]
where $\phi^j = \phi \circ \cdots \circ \phi$ is the $j$-th iterate of $\phi$. 

\begin{lemma}
  \label{lem:basic}
  For $j \geq 1$, $f_j \in \FF_q[T,x]$ and satisfies the following properties:
  \begin{itemize}
  \item $f_j$ is monic in $x$ of degree $q^j$;
  \item $f_j$ has degree $1 + q + \cdots + q^{j-1}$ in $T$;
  \item $\frac{\partial f_j}{\partial x} = (-1)^jT^{1 + q + \cdots + q^{j-1} - j}$; and
    \item Setting $T = 0$ in $f_j$ gives $\tilde f_j(x) = (x^q - x)^{q^{j-1}}$. 
  \end{itemize}
\end{lemma}

\begin{proof}
  All of these follow by induction and the formula
  \[
f_{j+1}(x) = T^{q^j} f_j\left(\phi(x)\right) \qedhere
  \]
\end{proof}

\begin{lemma}
  \label{lem:irreducible}
  For $j \geq 1$, $f_j$ is irreducible over $\FF_q(T)$, and its splitting field
  is totally $T$-adic.
\end{lemma}

\begin{proof}
The previous lemma shows that $f_j$ has $x$-degree $q^j$. Let
$\alpha_j$ be a root of $f_j$. We will show that
\[
\ord_\infty(\alpha_j) = -\frac{1}{q} - \frac{1}{q^2} - \cdots -\frac{1}{q^j}.
\]
It follows that $\alpha_j$ generates a totally ramified extension of degree
$q^j$ of the Laurent series field $\FF_q\Ls{T^{-1}}$, and hence $f_j$ is
irreducible over $\FF_q(T)$.

We proceed by induction. Since $f_1(x) = x^q - x - T$, the Newton
polygon for $f_1$ shows that $\ord_\infty(\alpha_1) = -1/q$.

Now suppose that $\ord_\infty(\alpha_j)$ is as expected. Since $f_{j+1}(x) =
T^{q^j}f_j(\phi(x))$, it follows that $\phi(\alpha_{j+1})$ is a root
of $f_j$. In particular, we may assume that $\phi(\alpha_{j+1}) =
\alpha_j$ in what follows. The Newton polygon of
\[
\phi(x) - \alpha_j = x^q / T - x / T - \alpha_j
\]
with respect to $\ord_\infty$ shows that every root has valuation
\[
(\ord_\infty(\alpha_j) - 1) / q = -\frac{1}{q} - \cdots - \frac{1}{q^{j+1}}.
\]
This completes the induction.

Now we prove that the splitting field of $f_j$ is totally $T$-adic. To that end,
it suffices to show that $f_j$ splits completely over $\FF_q\Ps{T}$. For $j =
1$, we see that $f_1(x) = x^q - x - T$. The elements of $\FF_q$ satisfy this
equation modulo~$T$; since its derivative is $f_1'(x) = -1$, Hensel's lemma
lifts each of these roots uniquely to $\FF_q\Ps{T}$. Now suppose that $f_j$
splits completely over $\FF_q\Ps{T}$ for some $j \geq 1$, and let $\alpha$ be a
root of $f_{j+1}$. As this element is constructed via iterated pre-images, there
is some root $\beta$ of $f_j$ such that $\phi(\alpha) = \beta$. Now $\alpha$
satisfies the equation $x^q - x - T\beta$. Exactly as in the case $j = 1$, each
element of $\FF_q$ satisfies this equation modulo~$T$, and each solution lifts
via Hensel's lemma. Hence, $\alpha \in \FF_q\Ps{T}$. This completes the proof.
\end{proof}

Theorem~\ref{thm:liminf} is an immediate consequence of
Theorem~\ref{thm:height_bound} and the following height estimate:

\begin{theorem}
  \label{thm:liminf-integral}
  Let $\alpha_j$ be a root of $f_j$. Then
\[
h(\alpha_j) = \frac{1}{q-1} + o(1) \text{ as $j \to \infty$}.
\]  
\end{theorem}

\begin{proof}
 Since $f_j$ is irreducible, monic, and has coefficients in $\FF_q[T]$, it is
 the minimal polynomial for $\alpha_j$. Thus, we see that
  \[
h(\alpha_j) = \frac{\deg_T(f_j)}{\deg_x(f_j)} = \frac{1 + q + \cdots + q^{j-1}}{q^j} = \frac{1}{q-1} + o(1). \qedhere
\]
\end{proof}


\section{The integral and unit cases}

Write $\cR_q$ for the set of functions $\alpha \in \cT_q$ that are
\textbf{$T$-adic integers}: $\ord_0(\beta) \geq 0$ for all Galois conjugates
$\beta$ of $\alpha$. Equivalently, $\alpha \in \cT_q$ is a $T$-adic integer if
and only if its minimal polynomial has the form
\[
f(x) = a_d x^d + \cdots + a_0 \qquad \text{with $a_i \in \FF_q[T]$ and $a_d(0) \ne
  0$}.
\]
(Note that the ring of $T$-adic integers contains all totally $T$-adic functions
that are integral over $\FF_q[T]$.)
The ultrametric inequality shows that $\cR_q$ is a subring of $\cT_q$. 

Write $\cR_q^\times$ for the unit group of $\cR_q$, which we refer to as
\textbf{$T$-adic units}. We claim that $\alpha \in \cR_q^\times$ if and only if
$\ord_0(\beta) = 0$ for all Galois conjugates $\beta$ of $\alpha$. Indeed, if
$\alpha \in \cR_q^\times$, then its inverse must lie in $\cR_q$, and so all of
its conjugates must have valuation~0. Conversely, suppose that all of the Galois
conjugates of $\alpha \in \cR_q$ have valuation~0 with respect to
$\ord_0$. Write $f$ for the minimal polynomial of $\alpha$. The Newton polygon
of $f$ with respect to $\ord_0$ is a single horizontal segment. The reversed
polynomial $f^*(x) = x^{\deg(f)}f(1/x)$ is the minimal polynomial of
$\alpha^{-1}$, and its Newton polygon must also be a horizontal segment. In
particular, the leading coefficient of $f^*$ has valuation~0, so we conclude
that $\alpha^{-1} \in \cR_q$. That is, $\alpha$ is a $T$-adic unit.

One can bound the heights of $T$-adic integers and units using a geometric
argument as in Theorem~\ref{thm:height_bound}, but it is slightly more
involved. Instead, we give an algebraic proof as in \S\ref{sec:min_pols}. 

\begin{theorem}
  \label{thm:Pottmeyer-integers-units}
  Let $q$ be a prime power. 
  \begin{enumerate}
    \item For any $\alpha \in \cR_q \smallsetminus \FF_q$, we have
  \[
  h(\alpha) \geq \frac{1}{q},
  \]
  and this inequality is sharp.
    \item For any $\alpha \in \cR_q^\times \smallsetminus \FF_q$, we have
  \[
  h(\alpha) \geq \frac{1}{q-1},
  \]
  and this inequality is sharp.
  \end{enumerate}
\end{theorem}

\begin{proof}
We begin with $\alpha \in \cR_q \smallsetminus \FF_q$. Let $r$ be
the number of Galois conjugates of $\alpha$ such that $\ord_0(\alpha) > 0$, and
let $d$ be the degree of $\alpha$ over $\FF_q(T)$. By
Lemma~\ref{lem:positive-valuation}, we find that
\[
h(\alpha) \geq \frac{r}{d}.
\]
Since $\alpha$ is a $T$-adic integer, we see that the remaining conjugates have
valuation zero. Then Lemma~\ref{lem:zero-valuation} shows that
\[
(q-1)h(\alpha) \geq \frac{d-r}{d}.
\]
Summing the two displayed inequalities gives $qh(\alpha) \geq 1$.

To see that this height bound is sharp, let $\alpha$ be a root of the polynomial
\[
f(x) = x^q - x + T.
\]
The case $j = 1$ of Lemma~\ref{lem:irreducible} shows that $f$ is
irreducible over $\FF_q(T)$ and that its splitting field is totally
$T$-adic. Since the leading coefficient of $f$ is~1, $\alpha \in \cR_q$,
and we conclude that
\[
h(\alpha) = \frac{\deg_T(f)}{\deg_x(f)} = \frac{1}{q}. 
\]

Now we turn to the case $\alpha \in \cR_q^{\times} \smallsetminus \FF_q$. Since
every conjugate of $\alpha$ has $T$-adic valuation zero,
Lemma~\ref{lem:zero-valuation} immediately gives the desired lower bound.

To see that this inequality is sharp, let $\alpha$ be a root of the polynomial
  \[
  f(x) = x^{q-1} - 1 + T.
  \]
Hensel's Lemma shows that $f$ splits completely over $\FF_q\Ps{T}$, and the
Newton polygon of $f$ with respect to $\ord_0$ shows that $\alpha \in
\cR_q^\times$.  Moreover, $f$ is irreducible because its Newton polygon with
respect to $\ord_\infty$ is a single segment with no lattice point in its
interior. In particular, $f$ is the minimal polynomial for $\alpha$, and we have
  \[
  h(\alpha) = \frac{\deg_T(f)}{\deg_x(f)} = \frac{1}{q-1}. \qedhere
  \]
\end{proof}

Now we turn to the question of bounding the liminf in the case of $T$-adic
integers and units.

\begin{theorem}
  \label{thm:liminf_int}
  Let $q$ be a prime power. Then 
  \[
  \frac{1}{q} \leq \liminf_{\alpha \in \cR_q} h(\alpha) \leq \frac{1}{q-1}.
  \]
\end{theorem}

\begin{proof}
  The lower bound is given by Theorem~\ref{thm:Pottmeyer-integers-units}. The
  sequence $(\alpha_j)$ constructed in Section~\ref{sec:small_height} consists
  of $T$-adic integers, so we may reuse the upper bound in
  Theorem~\ref{thm:liminf}.
\end{proof}

\begin{theorem}
  \label{thm:liminf_int_unit}
  Let $q$ be a prime power. Then
  \[
  \frac{1}{q-1} \leq \liminf_{\alpha \in \cR_q^{\times}} h(\alpha) \leq
  \begin{cases}
    \frac{1}{q-2} & \text{ if $q \ne 2$}
    \\
    2 & \text{ if $q = 2$}.
  \end{cases}
  \]
\end{theorem}

\begin{proof}
  The lower bound is given by Theorem~\ref{thm:Pottmeyer-integers-units}.

  For any $\beta \in \cR_q$, we find that $1 + T\beta \in
  \cR_q^\times$. Therefore,
  \[
  h(1+T\beta) \leq h(1) + h(T) + h(\beta) = 1 + h(\beta).
  \]
  Theorem~\ref{thm:liminf_int} shows the liminf over $\beta \in \cR_q$ is
  bounded above by $\frac{1}{q-1}$, so we get
  \[
  \liminf_{\alpha \in \cR_q^\times} h(\alpha) \leq 1 + \frac{1}{q-1}.
  \]
  In the case $q = 2$, this gives the upper bound~2, as desired.

  Now suppose that $q \ne 2$. Consider the polynomial
  \[
  \psi(x) = \frac{x^{q-1} - 1}{T}. 
  \]
  For $n \geq 1$, define
  \[
  g_n(x) = T^{1+(q-1)+\cdots+(q-1)^{n-1}} \left[\psi^n(x) - 1\right],
  \]
  where $\psi^n = \psi \circ \cdots \circ \psi$ is the $n$-th iterate of $\psi$.
  Just as in Section~\ref{sec:small_height}, one shows that $g_n$ lies in
  $\FF_q[T][x]$, that it is monic and irreducible, and that it splits completely
  over $\FF_q\Ps{T}$. Let $\beta_n$ be a root of $g_n$. The Newton polygon for
  $g_1$ shows that all of its roots have valuation~0 with respect to
  $\ord_0$. If we assume that all conjugates of $\beta_n$ have valuation~0, then
  the Newton polygon for $\psi(x) - \sigma(\beta_n)$ has a single horizontal
  segment for any Galois automorphism $\sigma$, and hence all conjugates of
  $\beta_{n+1}$ have valuation~0. By induction, $\beta_n \in \cR_q^\times$ for
  all $n$. Finally, one verifies that $h(\beta_n) = \frac{1}{q-2} + o(1)$ as $n
  \to \infty$. (All of this discussion holds when $q = 2$, except that
  $h(\beta_n) \to \infty$.) This establishes the upper bound in the theorem when
  $q \ne 2$.
\end{proof}


\bibliographystyle{plain}
\bibliography{small_height}
\end{document}